\documentclass[11pt,twoside]{article}
\usepackage{amsmath,amssymb,graphicx,url}
\pdfoutput=1 
\newtheorem{proposition}{Proposition}[section]
\newtheorem{theorem}[proposition]{Theorem}
\newtheorem{lemma}[proposition]{Lemma}
\newtheorem{corollary}[proposition]{Corollary}

\newtheorem{example}[proposition]{Example}
\newtheorem{algorithm}[proposition]{Algorithm}

\newenvironment{proof}{{\noindent \bf Proof:}}{\hfill $\fbox{}$ \vspace*{5mm}}
\newcommand{\ba}{{\bf a}}

\newcommand{\bc}{{\bf c}}
\newcommand{\be}{{\bf e}}

\newcommand{\bw}{{\bf w}}

\newcommand{\la}{\lambda}
\newcommand{\La}{\Lambda}

\newcommand{\ca}{\mathcal{A}}

\newcommand{\cd}{\mathcal{D}}

\newcommand{\co}{\mathcal{O}}

\newcommand{\cx}{\mathcal{X}}

\newcommand{\cy}{\mathcal{Y}}

\newcommand{\diag}{{\rm diag}}

\newcommand{\dist}{{\rm dist}}

\newcommand{\grad}{{\rm grad\;}}
\newcommand{\Hess}{{\rm Hess\;}}

\newcommand{\ve}{{\rm vec}}

\newcommand{\tr}{{\rm tr}}

\newcommand{\R}{{\mathbb R}}

\newcommand{\Rnn}{{\mathbb R}^{n\times n}}

\newcommand{\SRn}{\mathbb{SR}^{n\times n}}
\newcommand{\BE}{\begin{equation}}
\newcommand{\EE}{\end{equation}}

\newcommand{\normmm}[1]{{\vert\kern-0.25ex \vert\kern-0.25ex \vert #1
    \vert\kern-0.25ex \vert\kern-0.25ex\vert}}
\begin{document}

\title{A Preconditioned Riemannian Gauss-Newton Method for Least Squares Inverse Eigenvalue Problems}
\author{Teng-Teng Yao\thanks{Department of Mathematics, School of Sciences, Zhejiang University of Science and Technology, Hangzhou 310023, People's Republic of China (yaotengteng718@163.com). The research of this author is supported by the National Natural Science Foundation of China (No. 11701514).}
\and Zheng-Jian Bai\thanks{Corresponding author. School of Mathematical Sciences and Fujian Provincial Key Laboratory on Mathematical Modeling \& High Performance Scientific Computing,  Xiamen University, Xiamen 361005, People's Republic of China (zjbai@xmu.edu.cn). The research of this author is partially supported by the National Natural Science Foundation of China (No. 11671337), the Natural Science Foundation of Fujian Province of China (No. 2016J01035), and the Fundamental Research Funds for the Central Universities (No. 20720180008).}
 \and Xiao-Qing Jin\thanks{Department of Mathematics, University of Macau, Macao, People's Republic of China (xqjin@umac.mo). The research of this author is supported by the research grant MYRG2016-00077-FST from University of Macau.}
 \and Zhi Zhao\thanks{Department of Mathematics, School of Sciences, Hangzhou Dianzi University, Hangzhou 310018, People's Republic of China (zhaozhi231@163.com). The research of this author is supported by the National Natural Science Foundation of China (No. 11601112).}
 }

%\date{Revised: October 18, 2007}
\maketitle
%-----------------------------------------------------------------------
\begin{abstract}
This paper is concerned with the  least squares inverse eigenvalue problem of reconstructing a linear parameterized real symmetric matrix from the prescribed partial eigenvalues in the sense of least squares, which was originally proposed by Chen and Chu [SIAM J. Numer. Anal., 33 (1996), pp. 2417--2430].
We provide a Riemannian inexact Gausss-Newton method for solving the least squares inverse eigenvalue problem.  The global  and local convergence analysis of the proposed method is discussed. Also, a preconditioned conjugate gradient method with an efficient preconditioner is proposed for solving the Riemannian Gauss-Newton equation. Finally, some numerical tests, including an application in  the inverse Sturm-Liouville problem,  are reported to illustrate the efficiency of the proposed method.
\end{abstract}

\vspace{3mm}


\vspace{3mm}
%-----------------------------------------------------------

\section{Introduction}\label{sec1}
An inverse eigenvalue problem (IEP) aims to reconstruct a structured matrix from the prescribed spectral data.
Inverse eigenvalue problems (IEPs) arise in various applications such as structural dynamics,  vibration, inverse Sturm-Liouville problem, control design, geophysics,
nuclear spectroscopy and molecular spectroscopy, etc. For the existence theory, numerical methods and applications of general IEPs, one may refer to
\cite{C98,CG02,CG05,CEH12,D03,FM95,G04,X98} and references therein.

In this paper, we consider the following least squares inverse eigenvalue problem, which was originally given by Chen and Chu \cite{CC96}.

{\bf LSIEP I.} {\em Given $l+1$ real symmetric matrices $A_0,A_1,\ldots,A_l \in \Rnn$ and $m$ real numbers $\lambda^*_1 \leq \lambda^*_2 \leq \cdots \leq \lambda^*_m$ $(m\leq n)$, find a vector $\mathbf{c}=(c_1,\ldots,c_l)^T\in \mathbb{R}^{l}$ and a permutation $\sigma =\{\sigma_1,\sigma_2,\ldots,\sigma_m\}$ with $1 \leq \sigma_1 < \sigma_2 < \cdots < \sigma_m \leq n$ to minimize the function
\[
f(\mathbf{c},\sigma) : = \frac{1}{2} \sum_{i=1}^m(\lambda_{\sigma_i}(\mathbf{c}) - \lambda_i^*)^2,
\]
where the real numbers $\lambda_1(\mathbf{c})\leq \lambda_2(\mathbf{c}) \leq \cdots \leq \lambda_n(\mathbf{c})$
are the eigenvalues of the matrix $A(\mathbf{c})$ defined by
\[
A(\mathbf{c}):= A_0 + \sum\limits_{i=1}^{l} c_i A_i.
\]
}

This is a nonlinear least-squares problem, where the cost function $f(\mathbf{c},\sigma)$ is a function of
a continuous variable $\bc$ and  a discrete variable $\sigma$. This is a special kind of mixed
optimization problem, where the function $f(\mathbf{c},\sigma)$ is nondifferentiable when
the perturbation $\sigma$ is changed. For the LSIEP I, there exists an equivalent least-squares problem
defined on a product manifold. Let $\cd(m)$ and $\mathcal{O}(n)$ denote the set of all real diagonal
matrices of order $m$ and the set of all real $n\times n$ orthogonal matrices, respectively.
Define $\Lambda^*_m := {\rm diag} (\lambda^*_1, \lambda^*_2, \ldots, \lambda^*_m)$, where
$\diag(\ba)$ denotes a diagonal matrix with $\ba$ on its diagonal. Given a matrix $\Lambda\in \cd(n-m)$,
${\rm blkdiag}\left(\Lambda^*_m \;,\; \Lambda\right)$ denotes the block diagonal
matrix obtained from $\Lambda^*_m$ and $\Lambda$. Based on Theorem 3.1 in \cite{CC96},
the LSIEP I is equivalent to the following problem.

{\bf LSIEP II.} {\em
Given $l+1$ real symmetric matrices $A_0,A_1,\ldots,A_l \in \Rnn$ and  $m$ real numbers
$\lambda^*_1 \leq \lambda^*_2 \leq \cdots \leq \lambda^*_m$ $(m\leq n)$, find a vector
$\mathbf{c}\in \mathbb{R}^{l}$, an orthogonal matrix $Q\in \mathcal{O}(n)$, and
a diagonal matrix $\Lambda\in \cd(n-m)$ to minimize the function
\[
h(\mathbf{c},Q,\Lambda) : = \frac{1}{2} \| A(\mathbf{c}) -
Q {\rm blkdiag}\left(\Lambda^*_m,\; \Lambda\right) Q^T \|_F^2,
\]
where $\|\cdot \|_F$ denotes the Frobenius matrix norm.
}

The LSIEP II is a nonlinear least-squares problem defined on the product manifold
$\mathbb{R}^{l} \times \mathcal{O}(n) \times \cd(n-m)$.
To solve the LSIEP II, Chen and Chu \cite{CC96} proposed a lift and projection (LP) method.
This method is a modification of the alternating projection method to an affine space
and a Riemannian manifold. To solve the LSIEP I, Chen and Chu \cite{CC96} proposed a hybrid method called the LP-Newton method. The idea is that an initial guess  of $(\bc, \sigma)$ is obtained by using  the LP method to  the LSIEP II, then the Newton method is applied to the LSIEP I by fixing the value of $\sigma$. This method has fast local convergence while it requires an exact  guess value of $\sigma$ via the LP method. We note that the LP-Newton method works only for small problems since the forming of the Hessian matrix of $f(\mathbf{c},\sigma)$ is very expensive. If $m=l$, then there is no permutation $\sigma$ in the function $f(\mathbf{c},\sigma)$. In this case,  the LSIEP I becomes a continuous optimization
problem. For this special case, Wang and Vong \cite{WV13} proposed a Gauss-Newton-like method.

Optimization methods on smooth manifolds have been widely studied and applied to
various kinds of areas such as numerical linear algebra and dynamical systems (see for instance \cite{ABG07,AMS08,AM12,ADM02,CD91,CG98,HM94,S94} and references therein). Recently, some Riemannian optimization methods were proposed for solving nonlinear eigenvalue problems and inverse eigenvalue problems \cite{YBZC16,ZBJ15,ZBJ17,ZJB16}.
In this paper, we propose a Riemannian inexact Gauss-Newton method for solving the LSIEP II. In \cite{AMS08}, Absil et al.
proposed a Riemannian Gauss-Newton method for  solving nonlinear least squares problems
defined between Riemannian manifold and Euclidean space, where the convergence analysis was not discussed. In \cite{GLN07}, Gratton et al. gave some approximate Gauss-Newton methods
for solving nonlinear least squares problems defined on Euclidean  space.
Sparked by \cite{AMS08,GLN07}, we present an efficient
Riemannian inexact Gauss-Newton method for solving the LSIEP II. The global convergence and local convergence rate are also discussed. An effective preconditioner is proposed for solving the Riemannian Gauss-Newton equation via the conjugate gradient (CG) method \cite{GV96}. Finally, some numerical experiments, including an application in  the inverse Sturm-Liouville problem,  are reported to show the efficiency of the proposed method for solving the LSIEP II.

Throughout this paper, we use the following notation.  The symbol $A^T$ denotes the transpose
of a matrix $A$. The symbol ${\rm Diag}(M):= {\rm diag}(m_{11}, m_{22}, \ldots, m_{nn})$ denotes
a diagonal matrix containing the diagonal elements of an $n \times n$ matrix $M = [ m_{ij} ]$.  Let  $\mathbf{0}_{m\times n}$ be the $m\times n$ zero matrix and $\be_k$ be the $k$-th column of the identity matrix $I_n$ of order $n$.  Let $\Rnn$ and $\SRn$ be the set of
all $n$-by-$n$  real matrices and the set of all $n$-by-$n$ real symmetric matrices, respectively.  Denote by $\tr(A)$ the trace of a square matrix $A$. For two matrices $A,B\in \Rnn$, $[A,B]:=AB-BA$ mean the Lie Bracket of $A$ and $B$.
Let $\ve(A)$ be the vectorization of a matrix $A$, i.e., a column vector obtained by stacking
the columns of $A$ on top of one another. For two finite-dimensional vector spaces $\cx$ and $\cy$
equipped with a scalar inner product $\langle\cdot,\cdot\rangle$ and its induced norm
$\|\cdot\|$, let $\ca:\cx\to \cy$ be a linear operator and the adjoint operator of $\ca$
be denoted by $\ca^*$. The operator norm of $\ca$ is defined by
$\normmm{\ca}:=\sup\{ \| \ca x\| \ | \ x\in\cx\mbox{ with } \|x\|=1\}$.

The remainder of this paper is organized as follows. In section \ref{sec2} we propose a Riemannian inexact Gauss-Newton method for solving the LSIEP II. In section \ref{sec3} we establish the global convergence and local convergence rate of the proposed approach under some conditions. A preconditioner is also proposed for solving the Riemannian Gauss-Newton equation. Finally, we report some numerical tests
in section \ref{sec4} and give some concluding remarks in section \ref{sec5}.

\section{Riemannian inexact Gauss-Newton method}\label{sec2}
In this section, we present a Riemannian inexact Gauss-Newton method for solving the LSIEP II.
Define  an affine subspace and an isospectral manifold   by
\begin{eqnarray*}
&\mathcal{A}: =\big\{ A_0 + \sum_{i=1}^{l} c_i A_i \ | \  c_i \in \R, i=1,2,\ldots,l \big\}, &\\[2mm]
&\mathcal{M}(\Lambda^*_m): =  \big \{ X = Q {\rm blkdiag}\left(\Lambda^*_m, \Lambda\right) Q^T  \in \SRn \; | \;
Q\in \mathcal{O}(n),\; \Lambda \in \cd(n-m) \big\}.&
\end{eqnarray*}
We see that $\mathcal{M}(\Lambda^*_m)$ is the set of all  real $n\times n$ symmetric matrices whose spectrum
contains the $m$ real numbers $\lambda^*_1, \lambda^*_2, \ldots, \lambda^*_m$.
Thus, the LSIEP II has a solution such that $h(\mathbf{c},Q,\Lambda) =0$ if and only if $\ca\cap \mathcal{M}(\Lambda^*_m) \neq \emptyset$.

Let $H$ be a nonlinear mapping between $\mathbb{R}^{l} \times \mathcal{O}(n) \times \cd(n-m)$ and $\mathbb{SR}^{n\times n}$ defined by
\BE\label{NMF}
H(\mathbf{c}, Q, \Lambda) = A(\mathbf{c}) -Q {\rm blkdiag}\left(\Lambda^*_m \;,\Lambda\right) Q^T,
\EE
for all $(\mathbf{c}, Q, \Lambda)\in\mathbb{R}^{l} \times \mathcal{O}(n) \times \cd(n-m)$.
Then, the LSIEP II can be written  as the following minimization problem:
\BE\label{LSIEP}
\begin{array}{lc}
\min  & \displaystyle  h(\mathbf{c}, Q, \Lambda):=  \frac{1}{2} \|H(\mathbf{c}, Q, \Lambda)\|_F^2 \\[2mm]
\mbox{subject to (s.t.)} &  \quad  (\mathbf{c}, Q, \Lambda)\in \mathbb{R}^{l} \times \mathcal{O}(n) \times \cd(n-m).
\end{array}
\EE

Sparked by the ideas in \cite{AMS08,GLN07}, we propose a Riemannian inexact Gauss-Newton method for solving Problem (\ref{LSIEP}). We note that the dimension of the product manifold $\mathbb{R}^{l}\times \co(n)\times \cd(n-m)$ is given by
\[
\text{dim}(\mathbb{R}^{l} \times \mathcal{O}(n) \times \cd(n-m))  =  l + \frac{n(n-1)}{2} + n-m.
\]
If $l<m$, then
\[
\text{dim}( \mathbb{R}^{l} \times \mathcal{O}(n) \times \cd(n-m))  <  \text{dim}( \SRn).
\]
Therefore, the nonlinear equation $H(\mathbf{c}, Q, \Lambda) ={\bf 0}_{n\times n}$ is an over-determined matrix equation defined on  the product manifold $\mathbb{R}^{l} \times \mathcal{O}(n) \times \cd(n-m)$.

Notice that  $\mathbb{R}^{l} \times \mathcal{O}(n) \times \cd(n-m)$ is an embedded submanifold of $\mathbb{R}^{l} \times \Rnn \times \cd(n-m)$.  One may equip $\mathbb{R}^{l} \times \mathcal{O}(n) \times \cd(n-m)$ with the induced Riemannian metric:
\[
g_{(\mathbf{c},Q,\Lambda)}\big( (\xi_1, \eta_1, \tau_1), (\xi_2, \eta_2, \tau_2) \big):=
\tr(\xi_1^T\xi_2)+\tr(\eta_1^T\eta_2)+\tr(\tau_1^T\tau_2),
\]
for any $(\mathbf{c},Q,\Lambda) \in \mathbb{R}^{l} \times \mathcal{O}(n) \times \cd(n-m)$, and $(\xi_1, \eta_1, \tau_1), (\xi_2, \eta_2, \tau_2)\in  T_{(\mathbf{c},Q,\Lambda)}\mathbb{R}^{l} \times \mathcal{O}(n) \times \cd(n-m)$, and its induced norm $\|\cdot\|$.
The tangent space $T_{(\mathbf{c},Q,\Lambda)}\mathbb{R}^{l} \times \mathcal{O}(n) \times \cd(n-m)$ of $\mathbb{R}^{l} \times \mathcal{O}(n) \times \cd(n-m)$ at $(\mathbf{c},Q,\Lambda) \in \mathbb{R}^{l} \times \mathcal{O}(n) \times \cd(n-m)$, which is given by \cite[p.42]{AMS08}
\begin{eqnarray*}
&&T_{(\mathbf{c},Q,\Lambda)}\mathbb{R}^{l} \times \mathcal{O}(n) \times \cd(n-m) \\
&=& \big\{(\mathbf{r}, Q\Omega, U) \ | \ \Omega^T = -\Omega,
\;\mathbf{r}\in \mathbb{R}^{l}, \Omega \in \Rnn, U\in \cd(n-m)\big\}.
\end{eqnarray*}
Hence, $( \mathbb{R}^{l} \times \mathcal{O}(n) \times \cd(n-m),g)$ is a Riemannian product manifold.

A Riemannian Gauss-Newton method  for solving  Problem (\ref{LSIEP})
can be stated as follows. Given the current iterate $X_k:=(\mathbf{c}_k,Q_k,\La_k)\in \mathbb{R}^{l} \times \mathcal{O}(n) \times \cd(n-m)$, solve the normal equation
\BE\label{eq:SNMNEW1}
\begin{array}{l}
(\mathrm{D}H(X_k))^*\circ \mathrm{D}H(X_k)[\Delta X_k] = -(\mathrm{D}H(X_k))^*[H(X_k)],
\end{array}
\EE
for $\Delta X_k:=(\Delta \mathbf{c}_k, \Delta Q_k, \Delta \Lambda_k) \in T_{X_k}\mathbb{R}^{l} \times \mathcal{O}(n) \times \cd(n-m)$. Here,
\[
\mathrm{D}H(X_k): T_{(\mathbf{c},Q,\Lambda)}\mathbb{R}^{l} \times \mathcal{O}(n) \times \cd(n-m) \to T_{H(X_k)} \SRn
 \]
is the Riemannian differential of $H$ at the point $X_k$, which is given by
\BE\label{eq:LSIEPDF1}
\mathrm{D}H(X_k) [\Delta X_k]=(A(\Delta \mathbf{c}_k)-A_0) + [ Q_k\overline{\Lambda}_k Q_k^T ,\Delta Q_kQ_k^T ]
-(Q_kP)\Delta \Lambda_k (Q_kP)^T,
\EE
where
\BE\label{def:LAMP}
\overline{\Lambda}_k :={\rm blkdiag} (\Lambda^*_m \;,\Lambda_k) \quad \mbox{and} \quad P :=[\mathbf{0}_{ (n-m)\times m}, I_{n-m}]^T.
\EE
With respect to the  Riemannian metric $g$, the adjoint $(\mathrm{D}H(X_k))^* : T_{H(X_k)} \SRn$ $\to
T_{X_k}\mathbb{R}^{l} \times \mathcal{O}(n) \times \cd(n-m)$ of $\mathrm{D}H(X_k)$ is given by
\BE\label{eq:LSIEPDF2}
(\mathrm{D}H(X_k))^*[\Delta Z_k]=\Big(\mathbf{v}(\Delta Z_k), [ Q_k\overline{\Lambda}_k Q_k^T ,\Delta Z_k ]Q_k,
 -{\rm Diag}\big((Q_kP)^T\Delta Z_k (Q_kP)\big)\Big),
\EE
where
\BE\label{df:vec}
\mathbf{v}(\Delta Z):= \big(\tr(A_1^T \Delta Z), \tr(A_2^T \Delta Z), \ldots, \tr(A_l^T \Delta Z) \big)^T.
\EE

In addition, the Riemannian gradient of $h$ at  a point  $X:=(\mathbf{c},Q,\Lambda)\in\mathbb{R}^{l} \times \mathcal{O}(n) \times \cd(n-m)$ has the following form \cite[p.185]{AMS08}:
\begin{eqnarray}\label{RGR}
&&\mathrm{grad}\;h(X) = (\mathrm{D}H(X))^* [ H(X)]  \nonumber\\
&=&(\mathbf{v} A(\mathbf{c}) - Q\overline{\Lambda} Q^T), [Q\overline{\Lambda} Q^T,
A(\mathbf{c}) - Q\overline{\Lambda} Q^T]Q,   \nonumber\\
&& -{\rm Diag}((QP)^T (A(\mathbf{c}) - Q\overline{\Lambda} Q^T) (QP)) ).
\end{eqnarray}
Let $\nabla$ dentoe the Riemannian connection of  $\mathbb{R}^{l} \times \mathcal{O}(n) \times \cd(n-m)$.
By using  $(8.31)$ in \cite[p.185]{AMS08} we obtain
\BE\label{eq:TTT}
\begin{array}{c}
\nabla^2 h(X)[\xi_X,\eta_X] =\langle{\rm D} H(X)[\xi_X], {\rm D}H(X)[\eta_X]\rangle + \langle H(X),\nabla^2 H(X)[\xi_X,\eta_X]\rangle,
\end{array}
\EE
for all
$X:= (\mathbf{c},Q,\Lambda) \in \mathbb{R}^{l} \times \mathcal{O}(n) \times \cd(n-m)$ and
$\xi_X, \eta_X \in T_{X}\mathbb{R}^{l} \times \mathcal{O}(n) \times \cd(n-m)$,
where $\nabla^2 h$ is a $(0,2)$-tensor field and $\nabla^2 H(X)[\cdot,\cdot]=\big[\nabla^2 H_{ij}(X)[\cdot,\cdot] \big]\in\mathbb{SR}^{n\times n}$ \cite[p.109]{AMS08}. The Riemannian Hessian ${\rm Hess}\; h(X)$ at a point $X\in \mathbb{R}^{l} \times \mathcal{O}(n) \times \cd(n-m)$ is determined by \BE\label{eq:TTT2}
\nabla^2 h(X)[\xi_X,\eta_X] =\langle{\rm Hess}\; h(X)[\xi_X], \eta_X \rangle,
\EE
for all $\xi_X, \eta_X \in T_{X}\mathbb{R}^{l} \times \mathcal{O}(n) \times \cd(n-m)$. In particular, if $X_{*}$ is a solution of the equation $H(X)=\mathbf{0}_{n\times n}$, then we can obtain
\BE\label{eq:HDF}
{\rm Hess}\; h(X_{*}) = ({\rm D} H(X_{*}))^*\circ {\rm D} H(X_{*}).
\EE

Based on the above discussion, a Riemannian inexact Gauss-Newton method for solving Problem (\ref{LSIEP}) can be described as follows.

\begin{algorithm} \label{nm1}
{\rm (A Riemannian inexact Gauss-Newton method)}
\begin{description}
\item [{Step 0.}] Choose an initial point $ X_0 \in \mathbb{R}^{l} \times \mathcal{O}(n) \times \cd(n-m)$, $\beta, \eta_{\max}\in (0,1)$, $\sigma \in (0,\frac{1}{2})$. Let $k:=0$.

\item [{Step 1.}] Apply the CG method to finding an approximate solution $\Delta X_k \in T_{X_k} \mathbb{R}^{l} \times \mathcal{O}(n) \times \cd(n-m)$ of
    \BE\label{eq:le}
    (\mathrm{D}H(X_k))^*\circ \mathrm{D}H(X_k)[\Delta X_k] = -  {\rm grad}\,h(X_k)
    \EE
    such that
    \BE\label{eq:tol1}
    \|(\mathrm{D}H(X_k))^*\circ \mathrm{D}H(X_k)[\Delta X_k] + (\mathrm{D}H(X_k))^*[H(X_k)]\|\leq  \eta_k\| {\rm grad}\,h(X_k)\|
    \EE
    and
    \BE\label{eq:tol2}
    \langle {\rm grad}\,h(X_k) ,\Delta X_k \rangle \le -\eta_k\langle \Delta X_k ,\Delta X_k\rangle,
    \EE
    where $\eta_k:=\min\{\eta_{\max}, \| {\rm grad}\,h(X_k)\|\}$. If {\rm (\ref{eq:tol1})} and {\rm (\ref{eq:tol2})} are not attainable, then let
    \[
    \Delta X_k:=-\grad h(X_k).
    \]

\item [{Step 2.}] Let $l_k$ be the smallest nonnegative integer $l$ such that
         \BE\label{DES}
         h\big(R_{X_k}(\beta^{l} \Delta X_k)\big)- h( X_k)\le \sigma\beta^l \langle {\rm grad}\,h(X_k),\Delta X_k\rangle.
         \EE
         Set
         \[
         X_{k+1}:= R_{X_k}(\beta^{l_k} \Delta X_k).
         \]
\item [{Step 3.}] Replace $k$ by $k+1$ and go to {\rm \textbf{Step 1}}.
\end{description}
\end{algorithm}

We point out that, in \textbf{Step 2} of Algorithm \ref{nm1}, $R$ is a retraction on  $\mathbb{R}^{l} \times \mathcal{O}(n) \times \cd(n-m)$,
which takes the form of
\BE\label{eq:retraction1}
R_{X_k}(\Delta X_k) = \big(\mathbf{c}_k+\Delta {\mathbf{c}_k}, R^o_{Q_k}( \Delta Q_k),
\Lambda_k +\Delta\Lambda_k \big),
\EE
where $R^o$ is a retraction on $\mathcal{O}(n)$, which may be chosen as \cite[p.58]{AMS08}:
\[
R^o_Q(\eta_Q) =  {\rm qf} (Q + \eta_Q),  \quad  \eta_Q \in  T_Q\mathcal{O}(n).
\]
Here, ${\rm qf}(A)$ denotes the $Q$ factor of the QR decomposition of an invertible matrix $A\in \mathbb{R}^{n\times n}$ as $A=Q\widehat{R}$,
where $Q$ belongs to $\mathcal{O}(n)$ and $\widehat{R}$ is an upper triangular matrix with strictly positive diagonal elements.
For the retraction $R$ defined by (\ref{eq:retraction1}), there exist two scalars $\nu >0$ and $\mu_{\nu} >0$ such that \cite[p.149]{AMS08}
\BE\label{retr:bd}
\nu\| \Delta X \| \geq  {\rm dist}\big(X,R_X (\Delta X) \big),
\EE
for all $X\in \mathbb{R}^{l} \times \mathcal{O}(n) \times \cd(n-m)$ and
\BE\label{dx:nu}
\Delta X \in T_X\mathbb{R}^{l} \times \mathcal{O}(n) \times \cd(n-m)\quad \mbox{with}\quad
\| \Delta X  \|\leq \mu_{\nu},
\EE
where ``$\dist$" means the Riemannian distance on the Riemannian product manifold $( \mathbb{R}^{l} \times \mathcal{O}(n) \times \cd(n-m),g)$ \cite[p.46]{AMS08}.
Of course, one may choose other retractions on $\mathcal{O}(n)$ via polar decomposition, Givens rotation, Cayley transform,
exponential mapping, or singular value decomposition (see for instance \cite[p.58]{AMS08} and \cite{ZBJ15}).

%------------------------------
\section{Convergence Analysis}\label{sec3}
In this section, we establish the global convergence and local convergence rate of Algorithm \ref{nm1}.

%------------------------------
\subsection{Global Convergence}

For the  global convergence of Algorithm \ref{nm1}, we have the following result.
The proof follows from Theorem 4.1 in \cite{ZBJ15}. Thus we omit it here.

%--------------------------
\begin{theorem} \label{thm:gc}
Any accumulation point $X_*$ of the sequence $\{X_k\}$ generated by Algorithm {\rm \ref{nm1}}
is a stationary point of the cost function $h$ defined in  Problem {\rm (\ref{LSIEP})}.
\end{theorem}

The search directions $\{\Delta X_k\}$ generated by Algorithm {\rm \ref{nm1}} have the following property.
\begin{lemma}\label{pro:ctr}
Let $X_*$ be an accumulation point of the sequence $\{X_k\}$ generated by Algorithm {\rm \ref{nm1}}.
If ${\rm D}H(X_{*}): T_{X_*}\mathbb{R}^{l} \times \mathcal{O}(n) \times \cd(n-m)\to  T_{H(X_*)}\SRn$
is surjective, then there exist three constants $\bar{\rho}, d_1, d_2 > 0$ such that
for all $X_k\in B_{\bar{\rho}}(X_{*})$,
\[
d_1\, \|  {\rm grad}\,h(X_k) \| \leq \| \Delta X_k \| \leq  d_2\, \| {\rm grad}\,h(X_k)\|,
\]
where $B_{\bar{\rho}}(X_{*}):=\{X\in\mathbb{R}^{l} \times \mathcal{O}(n) \times \cd(n-m)\ | \ \dist(X,X_*)<\bar{\rho}\}$.
\end{lemma}
\begin{proof}
Since ${\rm D}H(X_{*})$ is surjective by hypothesis and $H$ is continuously differentiable,
there exist three positive scalars $\bar{\rho},\kappa_0,\kappa_1>0$ such that for all $X_k\in B_{\bar{\rho}}(X_{*})$,
$({\rm D}H(X_k))^*\circ {\rm D}H(X_k)$ is nonsingular, and
\BE\label{hinv:bd}
\normmm{({\rm D}H(X_k))^*\circ {\rm D}H(X_k)} \le \kappa_0,\qquad
\normmm{\big(({\rm D}H(X_k))^*\circ {\rm D}H(X_k)\big)^{-1}} \le \kappa_1.
\EE
Let
\[
T(X_k) := ({\rm D}H(X_k))^*\circ {\rm D}H(X_k)[ \Delta X_k ] +  {\rm grad}\,h(X_k).
\]
It follows from (\ref{eq:tol1}), (\ref{eq:tol2}), and (\ref{hinv:bd}) that for all $X_k\in B_{\bar{\rho}}(X_{*})$,
\begin{eqnarray*}%\label{est:dxj}
\|\Delta X_k
\| &=& \| \big(({\rm D}H(X_k))^*\circ {\rm D}H(X_k)\big)^{-1}
\big[ ({\rm D}H(X_k))^*\circ {\rm D}H(X_k)[ \Delta X_k ] \big]\| \\
&=& \| \big(({\rm D}H(X_k))^*\circ {\rm D}H(X_k)\big)^{-1}
\big[ T(X_k) -  {\rm grad}\,h(X_k)\big] \|\\
&\leq& \normmm{\big(({\rm D}H(X_k))^*\circ {\rm D}H(X_k)\big)^{-1}}\big( \| T(X_k) \| + \| {\rm grad}\,h(X_k)\|\big)\\
&\le& \kappa_1 (1+\eta_k)\|  {\rm grad}\,h(X_k)\| \le \kappa_1 (1+\eta_{\max})\|  {\rm grad}\,h(X_k)\|\\
&\equiv& d_2\|  {\rm grad}\,h(X_k)\|
\end{eqnarray*}
and
\begin{eqnarray*}
& &\|  {\rm grad}\,h(X_k) \|\\
&=& \| T(X_k) - ({\rm D}H(X_k))^*\circ {\rm D}H(X_k)[ \Delta X_{k_j}
 ]\|\\
&\le& \|  T(X_k) \| + \| ({\rm D}H(X_k))^*\circ {\rm D}H(X_k)[ \Delta X_k ]\| \nonumber \\
&\leq &  \eta_k \|   {\rm grad}\,h(X_k)  \| + \|  ({\rm D}H(X_k))^*\circ {\rm D}H(X_k)\|\cdot\| \Delta X_k\|\nonumber\\
&\leq &  \eta_{\max} \|   {\rm grad}\,h(X_k)  \| + \kappa_0 \|  \Delta X_k  \|,
\end{eqnarray*}
which implies that for all $X_k\in B_{\bar{\rho}}(X_{*})$,
\[
\|  \Delta X_k  \|\ge \frac{1-\eta_{\max}}{\kappa_0}\|  {\rm grad}\,h(X_k)  \|
\equiv  d_1\|  {\rm grad}\,h(X_k)  \|.
\]
This completes the proof.
\end{proof}

%------------------
For the local convergence of Algorithm \ref{nm1} related to an isolated local minima of $h$, we have the following result.
The proof follows from \cite[Proposition 1.2.5]{B03}.
\begin{lemma} \label{lem:nons}
Let $X_{*}$ be an accumulation point of the sequence $\{X_{k}\}$ generated by Algorithm {\rm \ref{nm1}}.
If ${\rm D}H(X_{*}): T_{X_*}\mathbb{R}^{l} \times \mathcal{O}(n) \times \cd(n-m)\to  T_{H(X_*)}\SRn$
is surjective and $X_*$ is an isolated local minimizer of $h$,
then the sequence $\{X_{k}\}$ converges to $X_{*}$.
\end{lemma}
\begin{proof}
By assumption, there exists a parameter $\hat{\rho}>0$ such that $X_{*}$ is the only stationary point of $h$ in the neighborhood $B_{\hat{\rho}}(X_{*})$ and
\BE\label{minvalueofh}
h(X) > h(X_*), \quad \forall X\neq X_*,\quad  X \in B_{\hat{\rho}}(X_{*}).
\EE
Since $h$ is continuously differentiable and $X_*$ is a stationary point of $h$, i.e., ${\rm grad}\; h(X_{*}) $ $= 0_{X_*}$,
we can obtain
\BE\label{eq:ggnorm}
\lim_{{\rm dist}(X,X_*) \to 0} {\rm grad}\, h(X) = 0_{X_*}.
\EE
From (\ref{minvalueofh}) and (\ref{eq:ggnorm}), there exists a positive scalar $0<\rho < \min\{\hat{\rho},\bar{\rho}\}$ such that
\BE\label{eq:sss}
h\big (X_{*} \big) \le h\big( X \big) \quad \mbox{and} \quad \|\grad h(X)\| < \mu_{\nu}, \quad \forall X\in B_{\rho}(X_{*}),
\EE
where $\mu_{\nu}$ is defined in (\ref{dx:nu}) and $\bar{\rho}$ is defined in Lemma \ref{pro:ctr}.

Let
\[
\phi(t) := \min\limits_{ \{X | t\leq  {\rm dist}(X,X_{*}) \leq \rho \} } \big\{ h(X) - h(X_{*}) \big\}, \quad \forall t\in [0,\rho].
\]
We note that $\phi$ is a monotonically nondecreasing function of $t$ and thus $\phi(t) >0$ for all $t\in (0,\rho]$.
Using (\ref{eq:ggnorm}), for any $ \epsilon \in (0,\rho]$, there exists a constant $r \in (0,\epsilon]$ such that
\BE\label{WAN}
{\rm dist} (X,X_{*}) < r \quad \Longrightarrow \quad {\rm dist} (X,X_{*}) + \nu d_2 \|{\rm grad}\, h(X)\| < \epsilon,
\EE
where $\nu$ is  defined in (\ref{retr:bd}). Define the open set
\[
S := \big\{  X \ | \ {\rm dist}(X,X_{*}) < \epsilon, \; h(X) < h(X_{*}) + \phi(r)   \big\}.
\]
We claim that if $X_k \in S$ for some $k$, then $X_{k+1} \in S$.
Indeed, by using the definitions of $\phi$ and $S$, if $X_k \in S$, then
\BE\label{spec}
\phi \big( {\rm dist} (X_k,X_{*} ) \big) \leq h(X_k) - h(X_{*}) < \phi(r),
\EE
which implies that ${\rm dist}(X_k,X_{*}) < r$ since $\phi$ is monotonically nondecreasing. From {\rm (\ref{WAN})},
\BE\label{eq:spec1}
{\rm dist} (X_k,X_{*} ) + \nu d_2 \| {\rm grad}\, h(X_k) \| < \epsilon.
\EE
On the other hand,  it follows from Lemma \ref{pro:ctr}, (\ref{retr:bd}), and (\ref{eq:sss}) that
\begin{eqnarray}\label{eq:spec2}
{\rm dist} (X_{k+1},X_{*}  )
&\leq& {\rm dist} (X_k,X_{*} )  +  {\rm dist} \big( X_{k+1} ,X_k \big) \nonumber\\
&=&  {\rm dist} (X_k,X_{*} )  +  {\rm dist} \big( R_{X_k}( \rho^{l_k} \Delta X_k) ,X_k \big) \nonumber\\
&\leq& {\rm dist} (X_k,X_{*} )  + \nu \rho^{l_k} \|  \Delta X_k  \|
\leq {\rm dist} (X_k,X_{*} )  + \nu  \|  \Delta X_k  \| \nonumber\\
&\leq& {\rm dist} (X_k,X_{*} )  + \nu d_2 \|  {\rm grad}\, h(X_k) \|.
\end{eqnarray}
Since $h(X_{k+1}) \leq h(X_k)$, it follows from (\ref{eq:spec1}) and (\ref{eq:spec2}) that
\[
{\rm dist}(X_{k+1},X_{*}) < \epsilon, \quad h(X_{k+1})- h(X_{*}) < \phi(r).
\]
Therefore, we have $X_{k+1} \in S$.

By induction, if $X_{\overline{k}} \in S$ for some $\overline{k}$, then $X_k \in S$ for all $k \geq \overline{k}$.
Since $X_*$ is an accumulation point of the sequence $\{X_k\}$, there exists a subsequence $\{X_{k_j}\}$
such that $\lim\limits_{j\to \infty}X_{k_j}=X_*$.
Then there exists an integer $k_{\bar{l}}$ such that $X_{k_{\bar{l}}}\in S$ and thus $X_k\in S$ for all $k \ge k_{\bar{l}}$.
Since $h(X_{k+1})<h(X_k)$ for all $k \geq k_{\bar{l}}$ and $\lim\limits_{k_j\to \infty}h(X_{k_j}) = h(X_*)$,
we can obtain
\BE\label{fconvergence}
\lim_{k\to \infty}h(X_k) = h(X_*).
\EE
Using (\ref{spec}) and (\ref{fconvergence}) we have $\lim\limits_{k\to \infty}\phi\big( {\rm dist} (X_k,X_{*}) \big)=0$.
Since $\phi$ is monotone nondecreasing, it follows that $\lim\limits_{k\to \infty} {\rm dist} (X_k,X_{*}) =0$ and thus $X_k \rightarrow X_{*}$. The proof is complete.
\end{proof}

%------------------
\subsection{Convergence rate}
In this section, we discuss the local convergence rate of Algorithm \ref{nm1}.
The pullbacks of $H$ and $h$ are defined as $\widehat{H} := H\circ R$ and $\widehat{h} := h\circ R$,
where $R$ is the retraction defined in (\ref{eq:retraction1}).
In addition, we use $\widehat{H}_X := H\circ R_X$ and $\widehat{h}_X := h\circ R_X$ to denote the restrictions of
$\widehat{H}$ and $\widehat{h}$ to the tangent space $T_X\mathbb{R}^{l} \times \mathcal{O}(n) \times \cd(n-m)$.
For the Riemannian gradient of $h$ and the gradient of its pull back $\widehat{h}$,
it holds that \cite[p.56]{AMS08}
\BE\label{GRAA}
{\rm grad}\, h(X) = {\rm grad}\, \widehat{h}_X(0_X),
\EE
for all $X \in \mathbb{R}^{l} \times \mathcal{O}(n) \times \cd(n-m)$.
For the differentials of $H$ and its pull back $\widehat{H}$,
we have
\BE\label{DFA}
{\rm D}H(X) = {\rm D}\widehat{H}_X(0_X),
\EE
for all $X \in \mathbb{R}^{l} \times \mathcal{O}(n) \times \cd(n-m)$.

For the stepsize $\beta^{l_k}$ in (\ref{DES}), we have the following result \cite{RW12}.
\begin{lemma}\label{lem:stepsize}
Let $X_*$ be an accumulation point of the sequence $\{X_k\}$ generated by Algorithm {\rm \ref{nm1}}.
If ${\rm D}H(X_{*}): T_{X_*}\mathbb{R}^{l} \times \mathcal{O}(n) \times \cd(n-m)\to  T_{H(X_*)}\SRn$
is surjective and $\|H(X_*)\|_F$ is sufficiently small, then for $k$ sufficiently large,
$l_k = 0$ satisfies {\rm (\ref{DES})}.
\end{lemma}

\begin{proof}
Let $\Delta X_k^{GN}$ denote the exact solution to {\rm (\ref{eq:le})}. Then we have
\begin{eqnarray}\label{dx:n2}
& &({\rm D}H(X_k))^*\circ {\rm D}H(X_k)[ \Delta X_k - \Delta X_k^{GN}  ] \nonumber\\[2mm]
&=&  {\rm grad}\,h(X_k) + ({\rm D}H(X_k))^*\circ {\rm D}H(X_k)[ \Delta X_k ].
\end{eqnarray}
From {\rm (\ref{eq:le})} and (\ref{DFA}) it follows that
\BE\label{NEWEQ}
({\rm D}\widehat{H}_{X_k}(0_{X_k}))^*[\widehat{H}_{X_k}(0_{X_k})] + ({\rm D}\widehat{H}_{X_k}(0_{X_k}))^*\circ {\rm D}\widehat{H}_{X_k}(0_{X_k})
[ \Delta X_k^{GN} ] = 0_{X_k}.
\EE

By hypothesis, ${\rm D}H(X_{*})$ is surjective. We know from (\ref{eq:TTT}) and (\ref{eq:TTT2}) that
if $\|H(X_*)\|_F$ is sufficiently small, then the Riemannian Hessian $\Hess h(X_*)$ is positive definite.
Thus $X_*$ is an isolated local minimizer of $h$.
We have $X_k\to X_*$ by Lemma \ref{lem:nons}. Thus, by using Lemma \ref{pro:ctr},
(\ref{eq:tol1}), (\ref{hinv:bd}), and (\ref{dx:n2}), we have for all $k$ sufficiently large,
\begin{eqnarray}\label{ORDER}
&& \| \Delta X_k - \Delta X_k^{GN}\|  \nonumber\\
& = & \big\| \big(({\rm D}H(X_k))^*\circ {\rm D}H(X_k)\big)^{-1} \big[ {\rm grad}\,h(X_k)+ ({\rm D}H(X_k))^*\circ {\rm D}H(X_k)[ \Delta X_k ]\big] \big\|   \nonumber\\
& \le &  \normmm{\big(({\rm D}H(X_k))^*\circ {\rm D}H(X_k)\big)^{-1}}\cdot\big\|  {\rm grad}\,h(X_k)+ ({\rm D}H(X_k))^*\circ {\rm D}H(X_k)[ \Delta X_k ] \big\|   \nonumber\\
& \le &  \kappa_1 \eta_k  \|  {\rm grad}\,h(X_k)\| \le \kappa_1 \| {\rm grad}\,h(X_k)\|^2 \nonumber\\
&  \le &   \frac{\kappa_1}{d_1^2} \| \Delta X_k \|^2.
\end{eqnarray}

In addition, the differential operator ${\rm D}\widehat{H}_{X}$ is Lipschitz-continuous at $0_X$ uniformly in a neighborhood of $X_*$. That is, there exist four scalars $\kappa_2,\kappa_3, \delta_1, \delta_2 > 0$, such that
\BE\label{HLP1}
\left\{
\begin{array}{c}
\|({\rm D}\widehat{H}_{X}(\xi_X))^*- ({\rm D}\widehat{H}_{X}(0_X))^* \| \leq \kappa_2\| \xi_X  \|,\\[2mm]
\|({\rm D}\widehat{H}_{X}(\xi_X))^*\circ {\rm D}\widehat{H}_{X}(\xi_X)-({\rm D}\widehat{H}_{X}(0_X))^*\circ {\rm D}\widehat{H}_{X}(0_X)\|\leq \kappa_3\|\xi_X \|,
\end{array}
\right.
\EE
for all $X\in B_{\delta_1}(X_*)$ and  $\xi_X \in B_{\delta_2}(0_X)$.
Let
\BE\label{def:GXK}
G(X_k) := \widehat{H}_{X_k}(\Delta X_k) - \widehat{H}_{X_k}(0_{X_k}) - \mathrm{D}\widehat{H}_{X_k}(0_{X_k})[\Delta X_k].
\EE
By using Corollary 3.3 in \cite{Coleman12}, we have
\BE\label{ESTGXK}
\|G(X_k)\| \leq \sup\limits_{\theta\in [0,1]} \| \mathrm{D}\widehat{H}_{X_k}(\theta \Delta X_k) - \mathrm{D}\widehat{H}_{X_k}(0_{X_k})\|
\cdot \|\Delta X_k\|.
\EE
From (\ref{def:GXK}), we obtain
\begin{eqnarray*}
& &\widehat{h}_{X_k}(\Delta X_k) = \frac{1}{2} \| \widehat{H}_{X_k}(\Delta X_k) \|^2
= \frac{1}{2} \| \widehat{H}_{X_k}(0_{X_k}) + \mathrm{D}\widehat{H}_{X_k}(0_{X_k})[\Delta X_k] + G(X_k)\|^2\\
&=& \frac{1}{2} \| \widehat{H}_{X}(0_{X_k})\|^2 + \big\langle \widehat{H}_{X_k}(0_{X_k}),
 \mathrm{D}\widehat{H}_{X_k}(0_{X_k})[\Delta X_k] \big\rangle
+ \frac{1}{2} \big\langle \mathrm{D}\widehat{H}_{X_k}(0_{X_k})[\Delta X_k], \mathrm{D}\widehat{H}_{X_k}(0_{X_k})[\Delta X_k] \big\rangle\\
& &+ \big\langle  \widehat{H}_{X_k}(0_{X_k}) + \mathrm{D}\widehat{H}_{X_k}(0_{X_k})[\Delta X_k], G(X_k)\big\rangle
+ \frac{1}{2} \|G(X_k)\|^2\\
&=& \widehat{h}_{X_k}(0_{X_k}) + \big\langle \widehat{H}_{X_k}(0_{X_k}),
 \mathrm{D}\widehat{H}_{X_k}(0_{X_k})[\Delta X_k] \big\rangle
+ \frac{1}{2} \big\langle(\mathrm{D}\widehat{H}_{X_k}(0_{X_k}))^* \mathrm{D}\widehat{H}_{X_k}(0_{X_k})[\Delta X_k], \Delta X_k \big\rangle\\
& &+ \big\langle  \widehat{H}_{X_k}(0_{X_k}) + \mathrm{D}\widehat{H}_{X_k}(0_{X_k})[\Delta X_k], G(X_k)\big\rangle
+ \frac{1}{2} \|G(X_k)\|^2.
\end{eqnarray*}
Using  (\ref{hinv:bd}), (\ref{GRAA}), (\ref{DFA}), (\ref{NEWEQ}),  (\ref{ORDER}), (\ref{HLP1}), (\ref{ESTGXK}),
and the above equality, we have for all $k$ sufficiently large,
\begin{eqnarray*}
& &h\big(R_{X_k}(\Delta X_k)\big) - h(X_k)
-\frac{1}{2}  \big\langle  {\rm grad}\,h(X_k), \Delta X_k \big\rangle  \\
&=&\widehat{h}_{X_k}(\Delta X_k) - \widehat{h}_{X_k}(0_{X_k})-
\frac{1}{2}\big \langle ({\rm D}\widehat{H}_{X_k}(0_{X_k}))^*[\widehat{H}_{X_k}(0_{X_k})], \Delta X_k   \big \rangle \\
&=&\frac{1}{2} \big\langle \widehat{H}_{X_k}(0_{X_k}), \mathrm{D}\widehat{H}_{X_k}(0_{X_k})[\Delta X_k] \big\rangle
+ \frac{1}{2} \big\langle(\mathrm{D}\widehat{H}_{X_k}(0_{X_k}))^* \mathrm{D}\widehat{H}_{X_k}(0_{X_k})[\Delta X_k], \Delta X_k \big\rangle\\
& &+ \big\langle  \widehat{H}_{X_k}(0_{X_k}) + \mathrm{D}\widehat{H}_{X_k}(0_{X_k})[\Delta X_k], G(X_k)\big\rangle
+ \frac{1}{2} \|G(X_k)\|^2 \\
&=&\frac{1}{2} \big \langle ({\rm D}\widehat{H}_{X_k}(0_{X_k}))^*[\widehat{H}_{X_k}(0_{X_k})] + ({\rm D}\widehat{H}_{X_k}(0_{X_k}))^*\circ {\rm D}\widehat{H}_{X_k}(0_{X_k})[\Delta X_k^{GN} ], \Delta X_k   \big \rangle\\
& & +  \frac{1}{2} \big \langle  ({\rm D}\widehat{H}_{X_k}(0_{X_k}))^*\circ {\rm D}\widehat{H}_{X_k}(0_{X_k})[\Delta X_k - \Delta X_k^{GN} ],
\Delta X_k  \big \rangle \\
& &+ \big\langle  \widehat{H}_{X_k}(0_{X_k}) + \mathrm{D}\widehat{H}_{X_k}(0_{X_k})[\Delta X^{GN}_k], G(X_k)\big\rangle
+ \frac{1}{2} \|G(X_k)\|^2\\
& &+ \big\langle \mathrm{D}\widehat{H}_{X_k}(0_{X_k})[\Delta X_k - \Delta X_k^{GN} ], G(X_k)\big\rangle\\
&\le& 0  +  \frac{1}{2} \normmm{({\rm D}\widehat{H}_{X_k}(0_{X_k}))^*\circ {\rm D}\widehat{H}_{X_k}(0_{X_k})}\cdot\|\Delta X_k
- \Delta X_k^{GN} ]\|\cdot\| \Delta X_k  \|\\
& &+  \|\widehat{H}_{X_k}(0_{X_k}) + \mathrm{D}\widehat{H}_{X_k}(0_{X_k})[\Delta X^{GN}_k]\|\cdot\| G(X_k)\|
+ \frac{1}{2} \|G(X_k)\|^2\\
& &+  \normmm{\mathrm{D}\widehat{H}_{X_k}(0_{X_k})}\cdot\|\Delta X_k - \Delta X_k^{GN} ]\|\cdot\| G(X_k)\|\\
&\leq & \frac{1}{2} \frac{\kappa_1\kappa_0}{d_1^2} \| \Delta X_k   \|^3
+  \kappa_2\|\widehat{H}_{X_k}(0_{X_k})\|_F\cdot\|\Delta X_k \|^2
+ \frac{1}{2} \kappa_2^2\|\Delta X_k \|^4\\
& &+  \frac{\kappa_1\kappa_2}{d_1^2}\normmm{\mathrm{D}\widehat{H}_{X_k}(0_{X_k})} \cdot\|\Delta X_k \|^4\\
&=& \kappa_2\|H(X_k)\|_F\cdot\|\Delta X_k \|^2 + \frac{1}{2} \frac{\kappa_1\kappa_0}{d_1^2} \| \Delta X_k   \|^3
+ \Big(\frac{1}{2} \kappa_2^2+\frac{\kappa_1\kappa_2}{d_1^2}\normmm{\mathrm{D}\widehat{H}_{X_k}(0_{X_k})} \Big)\|\Delta X_k \|^4.
\end{eqnarray*}
If $\|H(X_*)\|_F$ is sufficiently small, then $\|H(X_k)\|_F$ is small enough for all $k$ sufficiently large.
By using the above inequality, (\ref{DES}) holds with  $l_k = 0$ for all $k$ sufficiently large.
This completes the proof.
\end{proof}

We now establish the local convergence rate of Algorithm \ref{nm1}.
%----------------
\begin{theorem} \label{th:qc}
Let $X_*$ be an accumulation point of the sequence $\{X_k\}$ generated by Algorithm {\rm \ref{nm1}}.
If ${\rm D}H(X_{*}): T_{X_*}\mathbb{R}^{l} \times \mathcal{O}(n) \times \cd(n-m)\to  T_{F(X_*)}\SRn$is surjective
and $\|H(X_*)\|_F$ is sufficiently small, then the whole sequence $\{X_k\}$ converges to $X_*$ linearly. Furthermore, if $H(X_*)=\mathbf{0}_{n\times n}$, then the whole sequence $\{X_k\}$ converges to $X_*$ quadratically.
\end{theorem}
\begin{proof}
By hypothesis, ${\rm D}H(X_{*})$ is surjective. From (\ref{eq:TTT}) and (\ref{eq:TTT2}) it follows that
if $\|H(X_*)\|_F$ is sufficiently small, then  $\Hess h(X_*)$ is positive definite.
We have $X_k\to X_*$ from Lemma \ref{lem:nons}. By Lemma {\rm \ref{lem:stepsize}}, we have
$X_{k+1} = R_{X_k}(\Delta X_k)$ for all $k$ sufficiently large. Using Lemma 7.4.8 and Lemma 7.4.9
in \cite{AMS08}, there exist three scalars $\tau_0,\tau_1,\tau_2>0$ such that
for all $k$ sufficiently large,
\BE\label{SS}
\left\{
\begin{array}{l}
\tau_0{\rm dist}(X_k,X_*)\le\| {\rm grad}\, h(X_k) \| \le \tau_1{\rm dist}(X_k,X_*),\\[2mm]
\| {\rm grad}\, h(X_{k+1}) \| = \| {\rm grad}\, h\big(R_{X_k}(\Delta X_{k})\big) \|
\leq  \tau_2 \| {\rm grad}\, \widehat{h}_{X_k}(\Delta X_{k})  \|.
\end{array}
\right.
\EE
By using Taylor's formula we have for all $k$ sufficiently large,
\begin{eqnarray}\label{est:gxjdx}
{\rm grad}\, \widehat{h}_{X_k}(\Delta X_k)
&=& {\rm grad}\, \widehat{h}_{X_k}(0_{X_k})+({\rm D}\widehat{H}_{X_k}(0_{X_k}))^*\circ
 {\rm D}\widehat{H}_{X_k}(0_{X_k})[\Delta X_k ]\nonumber \\[2mm]
& &+{\rm Hess}\; \widehat{h}_{X_k}(0_{X_k})[ \Delta X_k ]-
({\rm D}\widehat{H}_{X_k}(0_{X_k}))^*\circ {\rm D}\widehat{H}_{X_k}(0_{X_k})[\Delta X_k ]\nonumber\\
& & +  \int^1_0 \big( {\rm Hess}\; \widehat{h}_{X_k}(t\Delta X_k ) - {\rm Hess}\; \widehat{h}_{X_k}(0_{X_k}) \big)[ \Delta X_k ]{\rm d}t.
\end{eqnarray}

Since $H$ is twice continuously differentiable, it follows from (\ref{eq:TTT}) and (\ref{eq:TTT2}) that
there exist two scalars $\kappa_4>0$ and $\delta_3>0$ such that for all $X\in B_{\delta_3}(X_*)$,
\BE\label{eq:HESS}
\normmm{{\rm Hess}\; \widehat{h}_{X_k}(0_{X_k})-({\rm D}\widehat{H}_{X_k}(0_{X_k}))^*\circ {\rm D}\widehat{H}_{X_k}(0_{X_k})}
\leq \kappa_4 \|H(X_k)\|_F.
\EE
Furthermore, the Hessian operator ${\rm Hess}\; \widehat{h}_{X}$ is Lipschitz-continuous at $0_X$
uniformly in a neighborhood of $X_*$, i.e., there exist three scalars $\kappa_5>0$, $\delta_4 >0$,
and $\delta_5 > 0$, such that for all $X\in B_{\delta_4}(X_*)$
and $\xi_X \in B_{\delta_5}(0_X)$, it holds that
\BE\label{HLP2}
\normmm{{\rm Hess}\, \widehat{h}_{X}(\xi_X)- {\rm Hess}\, \widehat{h}_{X}(0_X)}\leq \kappa_5\| \xi_X  \|.
\EE
In addition, $H$ is Lipschitz-continuous in a neighborhood of $X_*$, i.e.,
there exits two constants $L>0$ and $\delta_6 > 0$ such that for all $X,Y\in B_{\delta_6}(X_*)$,
\BE\label{HHLP}
\| H(X) - H(Y) \|_F \leq L{\rm dist}(X,Y).
\EE
From Lemma \ref{pro:ctr}, (\ref{eq:tol1}), (\ref{SS}), (\ref{est:gxjdx}), (\ref{eq:HESS}), and (\ref{HLP2}), we have
for $k$ sufficiently large,
\begin{eqnarray} \label{SSS}
&&\frac{\tau_0}{\tau_2}\, {\rm dist} (X_{k+1},X_*) \le \| {\rm grad}\, \widehat{h}_{X_k}(\Delta X_k) \| \nonumber\\[2mm]
&\leq& \big\| {\rm grad}\, \widehat{h}_{X_k}(0_{X_k})
+ ({\rm D}\widehat{H}_{X_k}(0_{X_k}))^*\circ {\rm D}\widehat{H}_{X_k}(0_{X_k})[\Delta X_k ] \big\| \nonumber\\[2mm]
& & + \big\|{\rm Hess}\; \widehat{h}_{X_k}(0_{X_k})[ \Delta X_k ] -
({\rm D}\widehat{H}_{X_k}(0_{X_k}))^*\circ {\rm D}\widehat{H}_{X_k}(0_{X_k})[\Delta X_k ] \big\| \nonumber\\[2mm]
& & + \bigg\| \int^1_0\big( {\rm Hess}\; \widehat{h}_{X_k}(t\Delta X_k ) -
{\rm Hess}\; \widehat{h}_{X_k}(0_{X_k})\big)[ \Delta X_k ] dt \bigg\| \nonumber\\[2mm]
&\leq& \|  {\rm grad}\,h(X_k)+({\rm D}\widehat{H}_{X_k}(0_{X_k}))^*\circ {\rm D}\widehat{H}_{X_k}(0_{X_k})[\Delta X_k ] \|\nonumber\\[2mm]
& &+ \kappa_4\|H(X_k)\|_F\cdot \|\Delta X_k\| + \kappa_5 \|\Delta X_k\|^2 \nonumber\\[2mm]
&\leq& \eta_k \| {\rm grad}\,h(X_k) \| + \kappa_4d_2\|H(X_k)\|_F\cdot\|  {\rm grad}\,h(X_k)\| \nonumber \\[2mm]
& &+ \kappa_5 d_2^2\|  {\rm grad}\,h(X_k)\|^2 \nonumber \\[2mm]
&\leq& \kappa_4d_2\tau_1\|H(X_k)\|_F{\rm dist}(X_k,X_{*}) + (1+\kappa_5 d_2^2)\| {\rm grad}\,h(X_k) \|^2 \nonumber\\[2mm]
&\leq& \kappa_4d_2\tau_1\|H(X_k)\|_F{\rm dist}(X_k,X_{*}) + ( 1 + \kappa_5 d_2^2)\tau_1^2\big({\rm dist}(X_k, X_*)\big)^2.
\end{eqnarray}
Thus,
\begin{eqnarray*}
{\rm dist}(X_{k+1}, X_* ) &\leq& \displaystyle \frac{\tau_1\tau_2}{\tau_0}\kappa_4d_2\|H(X_k)\|_F{\rm dist}(X_k,X_{*})
+ \frac{\tau_1^2\tau_2}{\tau_0}( 1 + \kappa_5 d_2^2)\big({\rm dist}(X_k, X_*)\big)^2\\
&=& c_1\|H(X_k)\|_F{\rm dist}(X_k,X_{*}) + c_2\big({\rm dist}(X_k, X_*)\big)^2,
\end{eqnarray*}
where $c_1:=\frac{\tau_1\tau_2}{\tau_0}\kappa_4d_2$ and $c_2 :=\frac{\tau_1^2\tau_2}{\tau_0}( 1 + \kappa_5 d_2^2)$.
If $\|H(X_*)\|_F$ is sufficiently small, then $\|H(X_k)\|_F$ is small enough such that
$c_1\|H(X_k)\|_F <1$ for all $k$ sufficiently large.
Thus if $\|H(X_*)\|$ is sufficiently small, then$\{X_k\}$ converges to $X_*$ linearly.

If $H(X_*)=\mathbf{0}_{n\times n}$, then we have from (\ref{HHLP}) for all $k$ sufficiently large,
\BE\label{HHLP1}
\| H(X_k) \|_F = \| H(X_k) - H(X_*)\|_F \leq L {\rm dist}(X_k,X_*).
\EE
Using (\ref{SSS}) and (\ref{HHLP1}), we have
\[
\begin{array}{rl}
{\rm dist}(X_{k+1}, X_* )\leq&  \displaystyle \frac{\tau_1\tau_2}{\tau_0}\big(\kappa_4d_2L
+ ( 1 + \kappa_5 d_2^2)\tau_1\big)\big({\rm dist}(X_k, X_*)\big)^2.
\end{array}
\]
Therefore, if $H(X_*)=\mathbf{0}_{n\times n}$, then $\{X_k\}$ converges to $X_*$ quadratically.
This completes the proof.
\end{proof}

As a direct consequence of {\rm (\ref{SS})} and {\rm (\ref{SSS})}, we have the following result.

\begin{corollary}\label{lem:s93-fff}
Let $X_*$ be an accumulation point of the sequence $\{X_k\}$ generated by Algorithm {\rm \ref{nm1}}. Suppose
the assumptions in Theorem {\rm \ref{th:qc}} are satisfied. Then there exists two constants $\mu_1, \mu_2>0$ such that
for all $k$ sufficiently large,
\[
\begin{array}{rcl}
 \| {\rm grad}\,  h(X_{k+1}) \| \leq \mu_1\|H(X_k)\|_F\| {\rm grad}\,  h(X_k) \| + \mu_2\| {\rm grad}\,  h(X_k) \|^2.
\end{array}
\]
Furthermore, if $H(X_*)=\mathbf{0}_{n\times n}$, then there exists a scalar $\bar{\nu}>0$ such that for all $k$ sufficiently large,
\[
\begin{array}{rcl}
 \| {\rm grad}\,  h(X_{k+1}) \| \leq \bar{\nu}\| {\rm grad}\,  h(X_k) \|^2.
\end{array}
\]
\end{corollary}

\subsection{Surjectivity condition}
In this section, we provide the surjectivity condition of $\mathrm{D} H(X_{*})$,
where $X_*=(\bc_*,Q_*,\La_*)$ is an accumulation point of the sequence $\{X_k\}$ generated
by Algorithm \ref{nm1}.
Based on (\ref{eq:LSIEPDF1}), $\mathrm{D} H(X_*)$ is surjective if and only if the following matrix equation
\BE\label{eq:matrixeq11}
\left\{
\begin{array}{l}
(A(\Delta \mathbf{c})-A_0) + Q_*\overline{\Lambda}_* Q_*^T\Delta QQ_*^T - \Delta Q\overline{\Lambda}_* Q_*^T -(Q_*P)\Delta \Lambda (Q_*P)^T = \mathbf{0}_{n\times n}, \\[2mm]
\hspace{3cm}\mbox{s.t.}\quad (\Delta \mathbf{c}, \Delta Q, \Delta \Lambda) \in
T_{X_*}\mathbb{R}^{l} \times \mathcal{O}(n) \times \cd(n-m)
\end{array}
\right.
\EE
has a unique  solution $(\Delta \mathbf{c}, \Delta Q, \Delta \Lambda)=(\mathbf{0}_{m},\mathbf{0}_{n\times n}, \mathbf{0}_{(n-m)\times (n-m)})\in T_{X_*}\mathbb{R}^{l} \times \mathcal{O}(n) \times \cd(n-m)$, where $\overline{\Lambda}_*:={\rm blkdiag}\left(\Lambda^*_m \;,\Lambda_* \right)$ and $P$ is defined in (\ref{def:LAMP}).

For $W\in \mathbb{R}^{n\times n}$, define $\widehat{\mathrm{vec}}(W) \in \mathbb{R}^{\frac{n(n-1)}{2}}$ by
\[
\widehat{\mathrm{vec}}(W)\Big({\frac{(j-1)(j-2)}{2}+i}\Big) := W_{ij}, \quad i<j, \quad j=2,\ldots,n.
\]
This shows that $\widehat{\mathrm{vec}}(W)$ is a column vector obtained by stacking the strictly upper triangular part of $W$.
For $\mathbf{w}\in \mathbb{R}^{\frac{n(n-1)}{2}}$, define $\widehat{\mathrm{skew}}(\mathbf{w}) \in \mathbb{R}^{n\times n}$ by
\[
\widehat{\mathrm{vec}} \Big(\widehat{\mathrm{skew}}(\mathbf{w})\Big) := \mathbf{w}, \quad
 \quad \widehat{\mathrm{vec}} \Big( \big(\widehat{\mathrm{skew}}(\mathbf{w})\big)^T \Big) := -\mathbf{w},
\]
and
\[
\big(\widehat{\mathrm{skew}}(\mathbf{w})\big)_{ii}=0, \quad \quad i=1,2,\ldots,n.
\]
We observe that $\widehat{\mathrm{skew}}(\mathbf{w})$ is a skew-symmetric matrix constructed from $\mathbf{w}$.
Therefore, $\widehat{\mathrm{vec}}$ and $\widehat{\mathrm{skew}}$ are a pair of inverse operators.
In addition, there exists a matrix $\widehat{P}\in \mathbb{R}^{n^2\times \frac{n(n-1)}{2}}$ such that
\BE\label{def:PHAT}
\mathrm{vec}\big(\widehat{\mathrm{skew}}(\mathbf{w})\big) = \widehat{P}\mathbf{w}
\EE
for all $\mathbf{w}\in \mathbb{R}^{\frac{n(n-1)}{2}}$.
Since $\Delta Q\in T_{Q_*} \mathcal{O}(n)$, there exists a skew-symmetric matrix $\Delta \Omega \in \mathbb{R}^{n\times n}$ such that
$\Delta Q = Q\Delta \Omega$. For $\Delta \Omega \in \mathbb{R}^{n\times n}$, it follows from (\ref{def:PHAT}) that
there exists a vector $\Delta \mathbf{v}\in \mathbb{R}^{\frac{n(n-1)}{2}}$
such that $\mathrm{vec}\big(\Delta \Omega\big) = \widehat{P}\Delta \mathbf{v}$. Thus, we have
\BE\label{DQ}
\begin{array}{rcl}
\mathrm{vec}(\Delta Q) = \mathrm{vec}(Q\Delta \Omega) = (I_n \otimes Q) \mathrm{vec}(\Delta \Omega)
= (I_n \otimes Q)\widehat{P} \Delta \mathbf{v},
\end{array}
\EE
where $``\otimes"$ means the Kronecker product.
Let $\widehat{A}$ be an $n^2\times l$ matrix defined by
\BE\label{def:AHAT}
\widehat{A}:= \big[ \mathrm{vec}(A_1), \mathrm{vec}(A_2),  \ldots, \mathrm{vec}(A_l) \big]\in \mathbb{R}^{n^2\times l}.
\EE
Since $\Delta \Lambda\in \cd(n-m)$, there exists a matrix $G\in \R^{(n-m)^2\times (n-m)}$
and a vector $\Delta \bw\in \R^{n-m}$ such that
\BE\label{Dw}
{\rm vec}(\Delta \Lambda) = G\Delta \bw.
\EE
Based on (\ref{DQ}), (\ref{def:AHAT}), and (\ref{Dw}), the vectorization of the matrix equation (\ref{eq:matrixeq11}) is given by
\BE\label{eq:matrixeq22}
\begin{array}{l}
\Big[ \widehat{A},\; (Q_*\otimes Q)(I_n\otimes\overline{\Lambda}-\overline{\Lambda}\otimes I_n)\widehat{P},\;
(QP)\otimes (QP)G \Big]
\left[
\begin{array}{c}
\Delta \mathbf{c}\\
\Delta \mathbf{v}\\
\Delta \bw
\end{array}
\right]
= \mathbf{0}_{n^2}.
\end{array}
\EE

Based on the above analysis, we have the following surjectivity condition of $\mathrm{D} H(X_{*})$.
\begin{theorem}\label{riediff:frk1}
Let $X_*=(\bc_*,Q_*,\La_*)$ be an accumulation point of the sequence $\{X_k\}$ generated
by Algorithm {\rm \ref{nm1}}. Then ${\rm D}H(X_*)$ is surjective if and only if the following matrix
\[%BE\label{eq:fullrank1}
\Big[ \widehat{A},\; (Q_*\otimes Q_*)(I_n\otimes\overline{\Lambda}_*-\overline{\Lambda}_*\otimes I_n)\widehat{P},\;
(Q_*P)\otimes (Q_*P)G \Big]
\]%EE
is of full rank.
\end{theorem}

\subsection{Preconditioning technique}
In this section, we propose a preconditioner for solving (\ref{eq:le}). Here we adapt a centered preconditioner \cite[p.279]{SAAD03}.
For the CG method, instead of solving (\ref{eq:le}), we solve  the following  preconditioned linear system
\[
\left\{
\begin{array}{l}
(\mathrm{D}H(X_k))^*\circ M_k^{-1} \circ \mathrm{D}H(X_k)
[  \Delta X_k ] =-(\mathrm{D}H(X_k))^*\circ M_k^{-1} [H(X_k)], \\[2mm]
\text{s.t.} \quad \Delta X_k
\in T_{(\mathbf{c},Q,\Lambda)}\mathbb{R}^{l} \times \mathcal{O}(n) \times \cd(n-m),
\end{array}
\right.
\]
where $M_k:T_{H(X_k)}\SRn \rightarrow T_{H(X_k)}\SRn$ is a self-adjoint and positive definite linear operator.

An efficient centered preconditioner $M_k$ may be defined by
\begin{eqnarray}\label{def:mk}
M_k[\Delta Z_k]
&:=& (A(\mathbf{v}(\Delta Z_k))-A_0) + \big[ Q_k\overline{\Lambda}_k Q_k^T, [ Q_k\overline{\Lambda}_k Q_k^T, \Delta Z_k] \big] \nonumber\\
& & + Q_kPP^TQ_k^T \Delta Z_k Q_kPP^TQ_k^T + \hat{t} \Delta Z_k,
\end{eqnarray}
for all $ \Delta Z_k\in T_{H(X_k)}\SRn$, where $\hat{t}>0$ is a given constant.
Using (\ref{eq:LSIEPDF1}) and (\ref{eq:LSIEPDF2}) we have
\begin{eqnarray*}
& &\big(\mathrm{D}H(X_k)\circ (\mathrm{D}H(X_k))^*+
\hat{t}\mathrm{Id}_{T_{H(X_k)} \SRn}\big)[\Delta Z_k ] \\
&=& (A(\mathbf{v}(\Delta Z_k))-A_0) + \big[ Q_k\overline{\Lambda}_k Q_k^T, [ Q_k\overline{\Lambda} _kQ_k^T, \Delta Z_k] \big]\\
&&+ Q_kP{\rm Diag}\big(P^TQ_k^T \Delta Z_k Q_kP\big)P^TQ_k^T + \hat{t} \Delta Z_k,
\end{eqnarray*}
for all $ \Delta Z_k\in T_{H(X_k)}\SRn$, where $\mathrm{Id}_{T_{H(X_k)} \SRn}$ means the identity mapping on $T_{H(X_k)} \SRn$.
This shows that
\[
M_k \approx \mathrm{D}H(X_k)\circ (\mathrm{D}H(X_k))^*+ \hat{t}\mathrm{Id}_{T_{H(X_k)} \SRn}.
\]
From (\ref{def:mk}) we see that for any $ \Delta Z_k\in T_{H(X_k)}\SRn$,
\begin{eqnarray*}
 &&\mathrm{vec}\big(M_k [\Delta Z_k]\big) \\
&=&\mathrm{vec}\big(A(\mathbf{v}(\Delta Z_k))-A_0\big)  + \big( (Q_kPP^TQ_k^T) \otimes  (Q_kPP^TQ_k^T) \big)\mathrm{vec}(\Delta Z_k) \\
& & + \big((Q_k \otimes Q_k)(I_n \otimes \overline{\Lambda}_k - \overline{\Lambda}_k \otimes I_n)^2(Q_k^T \otimes Q_k^T)+\hat{t}I_{n^2} \big)\mathrm{vec}(\Delta Z_k) \\
&=& (Q_k \otimes Q_k)\big( (I_n\otimes \overline{\Lambda}_k - \overline{\Lambda}_k \otimes I_n)^2
+ (PP^T)\otimes (PP^T) + \hat{t}I_{n^2} \big)(Q_k^T \otimes Q_k^T)\mathrm{vec}(\Delta Z_k) \\
&& +\widehat{A}\widehat{A}^T \mathrm{vec}(\Delta Z_k).
\end{eqnarray*}
Let
\[
\widehat{B}_k : = (Q_k \otimes Q_k)\big( (I_n \otimes \overline{\Lambda}_k - \overline{\Lambda}_k\otimes I_n)^2
+ (PP^T)\otimes (PP^T)+\hat{t}I_{n^2} \big)(Q_k^T \otimes Q_k^T).
\]
It is clear that $\widehat{B}_k$ is positive definite.
Thus, for any $\Delta Z_k\in T_{H(X_k)}\SRn$,
\[
\left\{
\begin{array}{l}
\mathrm{vec}\big(M_k [\Delta Z_k]\big) = \big(\widehat{B}_k + \widehat{A}\widehat{A}^T \big)\mathrm{vec}(\Delta Z_k),\\[2mm]
\mathrm{vec}\big(M_k^{-1} [\Delta Z_k]\big) = \big(\widehat{B}_k + \widehat{A}\widehat{A}^T \big)^{-1}\mathrm{vec}(\Delta Z_k).
\end{array}
\right.
\]
Note that
\[
\widehat{A}\widehat{A}^T = \sum_{i=1}^{l} \mathrm{vec}(A_i)\mathrm{vec}(A_i)^T.
\]
Thus the matrix $\widehat{A}\widehat{A}^T$ is a low rank matrix, i.e.,
$\mathrm{rank}(\widehat{A}\widehat{A}^T) \leq l$.
Let
\[
\widehat{M}_k : = \widehat{B}_k + \widehat{A}\widehat{A}^T.
\]
By assumption, $l<m\le n<n^2$, $\widehat{M}_k$ is a low rank perturbation of
$\widehat{B}_k$. By using the Sherman-Morrison-Woodbury formula \cite{H89}, we can obtain
\[
\widehat{M}_k^{-1} = \big(\widehat{B}_k + \widehat{A}\widehat{A}^T\big)^{-1}
= \widehat{B}_k^{-1} - \widehat{B}_k^{-1}\widehat{A}\big(I_l + \widehat{A}^T\widehat{B}_k^{-1}\widehat{A} \big)^{-1}
\widehat{A}^T\widehat{B}_k^{-1},
\]
where
\[
\widehat{B}_k^{-1} = (Q_k \otimes Q_k)\Big((I_{n}\otimes \overline{\Lambda}_k - \overline{\Lambda}_k \otimes I_{n})^2
+ (PP^T)\otimes (PP^T) +\widehat{t}I_{n^2}\Big)^{-1}(Q_k^T \otimes Q_k^T),
\]
which can be  computed easily. For any vector $\mathbf{x}\in \mathbb{R}^{n^2}$, the matrix-vector products $(Q_k \otimes Q_k)\mathbf{x}$
and $(Q_k^T \otimes Q_k^T)\mathbf{x}$ can be computed via
\[
(Q_k \otimes Q_k)\mathbf{x} = \mathrm{vec}(Q_k \widehat{X}Q_k^T)\quad \mbox{and} \quad
(Q_k^T \otimes Q_k^T)\mathbf{x} = \mathrm{vec}(Q_k^T \widehat{X}Q_k),
\]
where $\widehat{X}\in \Rnn$ is the matrix such that $\mathrm{vec}(\widehat{X})=\mathbf{x}$.
We conclude that the matrix-vector product $\widehat{M}_k^{-1}\mathbf{x}$ can be computed efficiently, where the main computational cost is to calculate the inverse of $(I_l + \widehat{A}^T\widehat{B}_k^{-1}\widehat{A} )\in \R^{l\times l}$.

\section{Numerical Experiments}\label{sec4}
In this section we report the numerical performance of Algorithm \ref{nm1} for solving Problem (\ref{LSIEP}). All the numerical tests are carried out by using {\tt MATLAB} 7.1 running on a workstation with a Intel Xeon CPU E5-2687W at  3.10 GHz and 32 GB of RAM. To illustrate the efficiency of our algorithm, we compare Algorithm \ref{nm1} with the LP-Newton method (LP-N) in \cite{CC96}.

In our numerical tests, we set $\beta=0.5$, $\eta_{\max}=0.01$, $\sigma=10^{-4}$, and $\hat{t}=10^{-5}$.
The largest number of iterations in Algorithm \ref{nm1} and the LP-Newton method is set to be $10^5$,
and the largest number of iterations in the CG method is set to be $n^3$. Let `{\tt CT.}', `{\tt IT.}',  `{\tt LP.}',  `{\tt NF.}', `{\tt NCG.}',  `{\tt Res.}',  `{\tt grad.}' , and `{\tt err-c.}' denote the averaged total computing time in seconds, the averaged number of outer Newton or Gauss-Newton iterations, the averaged number of LP iterations,
the averaged number of function evaluations, the averaged total number of inner CG iterations,
the averaged residual $\|H(X_k)\|_F$ or $\sqrt{2f(\bc_k,\sigma_*)}$,   the averaged residual
$\|\grad h(X_k)\|$ or $\|\grad f(\bc_k,\sigma_*)\|$, and the averaged relative error $\|\bc_k-\widehat{\bc}\|_\infty/\|\widehat{\bc}\|_\infty$ at the final iterates of
the corresponding algorithms, accordingly.

For the LP-Newton method, the stopping criterion for the LP step is set to be
\[
\|\bc_{k}-\bc_{k-1}\|_F < 10^{-3}
\]
and the stopping criterion for the Newton step is set to be
\[
 \|\grad f(\bc_k,\sigma_*)\|_F < \zeta,
\]
and the stopping criterion for Algorithm \ref{nm1} is set to be
\[
 \|\grad h(X_k)\|_F < \zeta,
\]
where $\zeta>0$ is the prescribed tolerance.

We consider the following three examples.

%---------------------
\begin{example}\label{ex:1} \protect{\bf \cite{CC96}}
We consider the LSIEP with $n=l=m=5$. Let
\[
A_0 =
\left[
\begin{array}{ccccc}
0 & -1 & 0 & 0 & 0\\
-1 & 0 & -1 & 0 & 0 \\
0 & -1 & 0 & -1 & 0\\
0 & 0 & -1 & 0 & -1 \\
0 & 0 & 0 &-1 & 0
\end{array}
\right],\quad
A_k = 4\be_k\be_k^T, \quad k = 1,2,\ldots,5.
\]
We choose $\{1, 1, 2, 3, 4\}$ as the prescribed spectrum.
\end{example}
%%---------------------
\begin{example}\label{ex:2}
We consider the Sturm-Liouville problem of the form{\rm :}
\BE\label{ex:sl}
-\frac{d^{2}y}{dx^{2}}+q(x)y=\la y,\quad 0\le x\le\pi,
\EE
where $q$ is a real, square-integrable function and the following Dirichlet boundary conditions are imposed
\[
y(0)=y(\pi)=0.
\]
By using the Rayleigh-Ritz method in {\rm \cite{H78}},  the Rayleigh quotient of {\rm (\ref{ex:sl})} is given by
\[
R\big(y(x)\big) = \frac{ \int^{\pi}_0 \big((y'(x))^2 + q(x)y(x)^2\big){\rm d}x}{ \int^{\pi}_0 y(x)^2 {\rm d}x}.
\]
Suppose that $y(x) = \sum\limits_{j=1}^n w_j \sin(jx)$. By simple calculation, we have
\[%\BE\label{rq}
R\big(y(x)\big) =  \frac{\sum\limits_{i=1}^n \sum\limits_{j=1}^n i\cdot j \cdot w_i\cdot w_j
\cdot \delta^i_j
+ \frac{2}{\pi}\cdot \sum\limits_{i=1}^n \sum\limits_{j=1}^n w_i\cdot w_j \int^{\pi}_0 q(x) \sin(ix) \sin(jx) {\rm d}x }
{\sum\limits_{i=1}^n \sum\limits_{j=1}^n w_i\cdot w_j \cdot \delta^i_j},
\]
i.e.,
\[
\begin{array}{rcl}
R\big(y(x)\big) =  \frac{ \bw^T A \bw }{\bw^T\bw},
\end{array}
\]
where $\bw := (w_1, w_2, \ldots, w_n)^T$ and the entries of the symmetric matrix $A=[a_{ij}]\in\Rnn$ are given by
\[
\begin{array}{rcl}
a_{ij}
&=&  i\cdot j \cdot \delta^i_j +
\frac{2}{\pi}\cdot \int^{\pi}_0 q(x) \sin(ix) \sin(jx) {\rm d}x\\[4mm]
&=&  i\cdot j \cdot \delta^i_j +
\frac{2}{\pi} \int^{\pi}_0 q(x) \frac{\cos\big((i-j)x\big) - \cos\big((i+j)x\big)}{2} {\rm d}x,
\end{array}
\]
for $i,j=1,2,\ldots,n$.
If $q(x)=2\sum\limits_{k=1}^l c_k \cos(2kx)$, then one has
\[
a_{ij}=  i\cdot j  \cdot \delta^i_j +
\sum_{k=1}^l c_k \cdot \big(\delta^{2k}_{|i-j|}  -\delta^{2k}_{i+j} \big), \quad i,j=1,2,\ldots,n.
\]

Let $T_k$ $(k=1,2,\ldots, n-1)$ and $H_k$ $(k=1,2, \ldots, 2n-1)$ be $n\times n$ real matrices generated by
the {\tt MATLAB} built-in functions {\tt toeplitz} and {\tt hankel}{\rm :}
\[
T_k = {\tt toeplitz}(\be_{k+1}),\quad k=1,2,\ldots, n-1
\]
and
\[
H_k =
\left\{
\begin{array}{ll}
{\tt hankel}(\be_k,{\bf 0}_n), & k=1,2,\ldots, n,\\[2mm]
{\tt hankel}({\bf 0}_n,\be_{k-n+1}), & k=n+1,n+2,\ldots,2n-1.
\end{array}
\right.
\]
Define
\[
A_0 = \diag(1,2^2, 3^2, \ldots, n^2),\quad
A_k=
\left\{
\begin{array}{ll}
T_{2k} - H_{2k-1}, &  1 \leq k \leq \min\Big\{l, \frac{n-1}{2}\Big\}, \\[2mm]
-H_{2k-1}, &   \min\Big\{l, \frac{n-1}{2}\Big\} < k \leq l.
\end{array}
\right.
\]
Then
\[
A = A_0+\sum_{k=1}^l c_kA_k\equiv A(\bc).
\]
To estimate  the first $l$ Fourier coefficients of the potential $q(x)$ defined by {\rm \cite{H78}}
\[
q(x) = \sum_{k=1}^{\infty}  \frac{192}{\pi^4}\frac{1}{k^4}\cos(2kx),
\]
we consider the LSIEP with above $\{A_k\}$ and the $n$ eigenvalues of $A(\widehat{\bc})$ as the prescribed spectrum
for varying $n=m$ and $l$, where the entries of $\widehat{\bc}$ are given by
\[
\hat{c}_k= \frac{192}{\pi^4}\frac{1}{k^4}, \quad k=1,2,\ldots,l.
\]
\end{example}
%---------------------

%---------------------
\begin{example}\label{ex:3}
We consider the LSIEP with varying $n$, $l$, and $m$. Let $\widehat{\mathbf{c}}\in \mathbb{R}^{l}$ be a random vector and
$A_0,A_1,\ldots ,A_l $ be $n\times n$ random symmetric matrices, which are generated by the {\tt MATLAB} built-in function {\tt randn}{\rm :}
\[
\widehat{\mathbf{c}} := {\tt randn}(l,1),\quad B_k : = {\tt randn} (n,n),
\quad A_k = \frac{1}{2}(B_k + B_k^T), \quad k=0,1,\ldots,l.
\]
We choose the $m$ smallest eigenvalues of $A(\widehat{\bc})$ as the prescribed partial spectrum.
\end{example}
%---------------------

For Algorithm \ref{nm1}  and the LP-Newton method in \cite{CC96},  the starting points are generated by
 the {\tt MATLAB} built-in function {\tt eig}:
\[%\BE\label{sp1}
\begin{array}{ll}
\big[ Q_0, \widetilde{\Lambda} \big] = \mbox{\tt eig}\,(A(\bc_0),{\rm 'real'}), \quad
&\Lambda_0 = \widetilde{\Lambda}(m+1:n).
\end{array}
\]
For Example \ref{ex:1}, $\bc_0$ is set to be
\[
\bc_0 = (0.6316, 0.2378, 0.9092, 0.9866, 0.5007)^T.
\]
For Example \ref{ex:2}, $\bc_0$ is set to be a zero vector. For Example \ref{ex:3}, $\bc_0$  is formed by chopping the components of $\widehat{\bc}$ to two decimal places for  $n<100$ and  to three decimal places for $n\geq 100$.

We first apply the LP-Newton method and Algorithm \ref{nm1} to Example \ref{ex:1} with $\zeta=10^{-7}$. Both methods converge to the same least squares solution:
\[
\bc_*=(0.4423,0.6044,0.6566,0.6044,0.4423)^T
\]
and the spectrum of $A(\bc_*)$ is $\{0.5888,1.0422,2.0742,3.1446,4.1501\}$.

Table \ref{table1} lists numerical results for Example \ref{ex:1}. We see from Table \ref{table1} that Algorithm \ref{nm1} is not as effective as the LP-Newton method since $\|H(X_*)\|_F=0.4688$ is not small enough.

%----------------------------
\begin{table}[ht]\renewcommand{\arraystretch}{1.2} \addtolength{\tabcolsep}{-2pt}
  \caption{Comparison results for Example \ref{ex:1}.}\label{table1}
  \begin{center} {\scriptsize
   \begin{tabular}[c]{|c|c|c|c|c|c|l|l|}
     \hline
Alg.  & {\tt CT.} & {\tt LP.} & {\tt IT.} & {\tt NF.} &  {\tt NCG.}   &  {\tt Res.}   &  {\tt grad.}  \\  \hline
LP-N                     &  0.0320 s & 29 & 2  &  3  &  4.5 & $0.4688$ & $1.12\times 10^{-8}$   \\
Alg. \ref{nm1} with CG   &  0.6880 s &    & 635&2776 & 18.3 & $0.4688$ & $8.99\times 10^{-8}$  \\
Alg. \ref{nm1}  with PCG & 0.3280 s  &    & 655&2855 &  1.4 & $0.4688$ & $9.83\times 10^{-8}$  \\ \hline
  \end{tabular} }
  \end{center}
\end{table}
%----------------------------

We now apply the LP-Newton method and Algorithm \ref{nm1} to Example \ref{ex:2} with $\zeta=10^{-8}$. Table \ref{table2} displays numerical results for Example \ref{ex:2}. We observe that Algorithm \ref{nm1} works much better than the LP-Newton method in terms of computing time.  We also see that the proposed preconditioner is very efficient.

%----------------------------
\begin{table}[ht]\renewcommand{\arraystretch}{1.2} \addtolength{\tabcolsep}{-2pt}
  \caption{Comparison results for Example \ref{ex:2}.}\label{table2}
  \begin{center} {\scriptsize
   \begin{tabular}[c]{|c|c|c|c|c|c|c|l|l|l|}
     \hline
Alg. & $(n,l,m)$  & {\tt CT.} & {\tt LP.} & {\tt IT.} & {\tt NF.} &  {\tt NCG.}   &  {\tt Res.}   &  {\tt grad.}  & {\tt err-c.} \\  \hline
               & (10,  6, 10)   &  0.1560 s  &  53 & 2 &  3  &   6   & $7.98\times 10^{-11}$ & $1.05\times 10^{-10}$ & $2.44\times 10^{-11}$  \\
               & (20, 12, 20)   &  1.4630 s  & 101 & 2 &  3  &   7   & $1.14\times 10^{-9}$  & $1.45\times 10^{-9}$  & $3.31\times 10^{-10}$  \\
LP-N           & (30, 18, 30)   &  8.4830 s  & 144 & 2 &  3  &   7   & $4.57\times 10^{-9}$  & $5.81\times 10^{-9}$  & $1.35\times 10^{-9}$ \\
               & (40, 22, 40)   &  30.109 s  & 181 & 3 &  4  &   7   & $3.77\times 10^{-12}$ & $2.56\times 10^{-12}$ & $1.69\times 10^{-12}$   \\
               & (50, 34, 50)   & 02 m 01 s  & 216 & 3 &  4  &   7   & $6.29\times 10^{-12}$ & $5.35\times 10^{-12}$ & $1.46\times 10^{-12}$  \\
               \cline{2-9}\hline
               & (10,  6, 10)   &  0.0450 s  &     & 6 &  7  &  30.3 & $4.94\times 10^{-14}$ & $1.41\times 10^{-13}$ & $2.01\times 10^{-14}$  \\
Alg. \ref{nm1} & (20, 12, 20)   &  0.0600 s  &     & 6 &  7  & 133.7 & $1.16\times 10^{-10}$ & $1.97\times 10^{-10}$ & $4.44\times 10^{-11}$  \\
with           & (30, 18, 30)   &  0.1830 s  &     & 6 &  7  & 355.2 & $3.48\times 10^{-11}$ & $1.99\times 10^{-10}$ & $9.40\times 10^{-12}$  \\
CG             & (40, 22, 40)   &  0.4890 s  &     & 6 &  7  & 700.8 & $3.22\times 10^{-12}$ & $3.05\times 10^{-11}$ & $1.57\times 10^{-12}$   \\
               & (50, 34, 50)   &  0.9690 s  &     & 6 &  7  & 1164  & $4.68\times 10^{-12}$ & $4.87\times 10^{-11}$ & $1.26\times 10^{-12}$ \\
               \cline{2-9}\hline
               & (10,  6, 10)   &  0.0420 s  &     & 5 &  6  &  1.2  & $5.56\times 10^{-14}$ & $1.69\times 10^{-13}$ & $2.83\times 10^{-14}$   \\
Alg. \ref{nm1} & (20, 12, 20)   &  0.0120 s  &     & 5 &  6  &  1.2  & $2.92\times 10^{-13}$ & $9.67\times 10^{-13}$ & $7.19\times 10^{-14}$   \\
with           & (30, 18, 30)   &  0.0080 s  &     & 5 &  6  &  1.2  & $1.32\times 10^{-12}$ & $2.02\times 10^{-12}$ & $7.17\times 10^{-13}$  \\
PCG            & (40, 22, 40)   &  0.0100 s  &     & 5 &  6  &  1.2  & $2.81\times 10^{-12}$ & $4.85\times 10^{-12}$ & $1.35\times 10^{-12}$  \\
               & (50, 34, 50)   &  0.0180 s  &     & 5 &  6  &  1.2  & $3.72\times 10^{-12}$ & $1.13\times 10^{-11}$ & $1.49\times 10^{-12}$ \\
               \cline{2-9}\hline
  \end{tabular} }
  \end{center}
\end{table}
%----------------------------

Next, we apply the LP-Newton method and Algorithm \ref{nm1} to Example \ref{ex:3}  with $\zeta=10^{-8}$. For comparison purposes,  we repeat our experiments over $10$ different problems. Table \ref{table3} shows numerical results for Example \ref{ex:3}.
We observe from Table \ref{table3} that Algorithm \ref{nm1} is more effective than the LP-Newton method in terms of computing time. We also see that the proposed preconditioner can reduce the number of inner CG iterations effectively.

To further illustrate the efficiency of Algorithm \ref{nm1}, we apply Algorithm \ref{nm1} with the proposed preconditioner to Examples \ref{ex:2}--\ref{ex:3} for varying $n,l,m$. The corresponding numerical results are displayed in Tables \ref{table5}--\ref{table6}.
We see from  Tables \ref{table5}--\ref{table6} that Algorithm \ref{nm1} with the proposed preconditioner works very efficient for different values of $n,l,m$.  Finally, the quadratic convergence of Algorithm \ref{nm1} is observed from Figure \ref{fig}, which agrees with our prediction.

%----------------------------
\begin{table}[ht]\renewcommand{\arraystretch}{1.2} \addtolength{\tabcolsep}{-2pt}
  \caption{Comparison results for Example \ref{ex:3}.}\label{table3}
  \begin{center} {\scriptsize
   \begin{tabular}[c]{|c|c|c|c|c|c|c|l|l|l|}
     \hline
Alg. & $(n,l,m)$  & {\tt CT.} & {\tt LP.} & {\tt IT.}& {\tt NF.} &  {\tt NCG.}   &  {\tt Err.}   &  {\tt Res.}  & {\tt err-c.} \\  \hline
               & (10,  6,  8)   &  0.0794 s & 3.3 & 3.1 & 4.1 &  60 & $4.22\times 10^{-11}$ & $6.15\times 10^{-11}$ & $2.28\times 10^{-11}$\\
               & (20, 10, 18)   &  1.1105 s & 3.5 & 3.0 & 4.0 &  40 & $4.89\times 10^{-11}$ & $1.51\times 10^{-10}$ & $1.14\times 10^{-11}$  \\
LP-N           & (30, 16, 25)   &  6.6156 s & 5.0 & 3.0 & 4.0 &  68 & $3.74\times 10^{-10}$ & $1.31\times 10^{-9}$  & $4.30\times 10^{-11}$   \\
               & (40, 20, 32)   &  22.379 s & 5.5 & 3.2 & 4.2 &  85 & $3.57\times 10^{-10}$ & $1.92\times 10^{-9}$  & $4.43\times 10^{-11}$ \\
               & (50, 34, 42)   &02 m 29 s  & 6.4 & 4.4 & 5.4 & 392 & $1.01\times 10^{-10}$ & $5.23\times 10^{-10}$ & $5.99\times 10^{-12}$   \\
               \cline{2-9}\hline
               & (10,  6,  8)   & 0.0150 s  &     & 3.9 & 4.9 & 676 & $5.31\times 10^{-11}$ & $5.78\times 10^{-10}$ & $1.80\times 10^{-12}$ \\
Alg. \ref{nm1} & (20, 10, 18)   & 0.0890 s  &     & 4.2 & 5.2 & 919 & $1.16\times 10^{-12}$ & $2.60\times 10^{-12}$ & $9.49\times 10^{-14}$  \\
with           & (30, 16, 25)   & 0.2763 s  &     & 4.7 & 5.7 &2130 & $2.75\times 10^{-11}$ & $9.27\times 10^{-11}$ & $1.75\times 10^{-12}$ \\
CG             & (40, 20, 32)   & 1.0520 s  &     & 8.0 &41.7 &3799 & $1.98\times 10^{-12}$ & $4.49\times 10^{-11}$ & $3.27\times 10^{-14}$   \\
               & (50, 34, 42)   & 1.8564 s  &     & 6.7 &17.6 &15517& $7.69\times 10^{-11}$ & $7.04\times 10^{-10}$ & $5.76\times 10^{-12}$   \\
               \cline{2-9}\hline
               & (10,  6,  8)   & 0.0031 s  &     & 3.0 & 4.0 &21.7 & $1.47\times 10^{-12}$ & $8.03\times 10^{-12}$ & $1.26\times 10^{-13}$   \\
Alg. \ref{nm1} & (20, 10, 18)   & 0.0117 s  &     & 3.0 & 4.0 & 8.7 & $2.16\times 10^{-10}$ & $6.09\times 10^{-10}$ & $9.17\times 10^{-13}$  \\
with           & (30, 16, 25)   & 0.0137 s  &     & 3.5 & 4.5 &  18 & $8.18\times 10^{-11}$ & $6.14\times 10^{-10}$ & $1.70\times 10^{-13}$ \\
PCG            & (40, 20, 32)   & 0.0110 s  &     & 3.5 & 4.5 &18.5 & $2.07\times 10^{-10}$ & $2.46\times 10^{-9}$  & $3.50\times 10^{-13}$  \\
               & (50, 34, 42)   & 0.0203 s  &     & 4.0 & 5.0 &43.2 & $2.87\times 10^{-13}$ & $1.01\times 10^{-11}$ & $7.57\times 10^{-15}$ \\
               \cline{2-9}\hline
  \end{tabular} }
  \end{center}
\end{table}
%----------------------------

%----------------------------
\begin{table}[ht]\renewcommand{\arraystretch}{1.2} \addtolength{\tabcolsep}{-2pt}
  \caption{Numerical results for Example \ref{ex:2}.}\label{table5}
  \begin{center} {\scriptsize
   \begin{tabular}[c]{|c|c|c|c|c|l|l|l|}
     \hline
$(n,l,m)$  & {\tt CT.}& {\tt IT.}   & {\tt NF.} &  {\tt NCG.}   &  {\tt Res.}   &  {\tt grad.}  &  {\tt err-c.}\\  \hline
(100,  50, 100) & 0.0920 s  &  5 &  6  & 1.2 & $3.24\times 10^{-11}$ & $5.79\times 10^{-11}$ & $1.49\times 10^{-11}$  \\
(200,  60, 200) & 0.3800 s  &  5 &  6  & 1.2 & $2.46\times 10^{-10}$ & $3.42\times 10^{-10}$ & $6.45\times 10^{-11}$   \\
(300,  70, 300) & 1.0400 s  &  5 &  6  & 1.2 & $5.73\times 10^{-10}$ & $9.52\times 10^{-10}$ & $1.61\times 10^{-10}$ \\
(400,  75, 400) & 2.3570 s  &  5 &  6  & 1.2 & $1.78\times 10^{-9}$  & $1.87\times 10^{-9}$  & $4.00\times 10^{-10}$ \\
(500,  80, 500) & 3.8170 s  &  5 &  6  & 1.2 & $2.97\times 10^{-9}$  & $3.14\times 10^{-9}$  & $4.01\times 10^{-10}$ \\
(600,  85, 600) & 5.8630 s  &  5 &  6  & 1.2 & $4.76\times 10^{-9}$  & $4.74\times 10^{-9}$  & $9.18\times 10^{-10}$ \\
(800,  90, 800) & 12.032 s  &  5 &  6  & 1.2 & $1.12\times 10^{-8}$  & $9.54\times 10^{-9}$  & $2.29\times 10^{-9}$ \\
\cline{2-8}\hline
  \end{tabular} }
  \end{center}
\end{table}

%----------------------------
\begin{table}[ht]\renewcommand{\arraystretch}{1.2} \addtolength{\tabcolsep}{-2pt}
  \caption{Numerical results for Example \ref{ex:3}.}\label{table6}
  \begin{center} {\scriptsize
   \begin{tabular}[c]{|c|c|c|c|c|c|l|l|l|}
     \hline
$(n,l,m)$  & {\tt CT.}& {\tt IT.}   & {\tt NF.} &  {\tt NCG.}   &  {\tt Err.}   &  {\tt Res.}   &  {\tt err-c.}\\  \hline
(100,  60,  80) &  0.0905 s & 3.0 & 4.0 & 27 & $3.04\times 10^{-11}$ & $4.17\times 10^{-10}$ & $3.73\times 10^{-14}$ \\
(200, 120, 160) &  0.6575 s & 3.0 & 4.0 & 33 & $6.86\times 10^{-11}$ & $1.85\times 10^{-9}$  & $4.00\times 10^{-14}$   \\
(300, 160, 200) &  3.3588 s & 3.7 & 4.7 & 64 & $2.14\times 10^{-11}$ & $1.43\times 10^{-9}$  & $3.71\times 10^{-14}$  \\
(400, 220, 280) &  10.702 s & 4.0 & 5.0 & 65 & $9.42\times 10^{-12}$ & $2.51\times 10^{-9}$  & $4.19\times 10^{-14}$  \\
(500, 340, 400) &  22.921 s & 4.0 & 5.0 & 48 & $2.48\times 10^{-11}$ & $5.64\times 10^{-9}$  & $6.29\times 10^{-14}$  \\
(600, 420, 480) &  46.968 s & 4.2 & 5.2 & 54 & $2.13\times 10^{-11}$ & $9.58\times 10^{-9}$  & $7.95\times 10^{-14}$  \\
 \cline{2-7}\hline
  \end{tabular} }
  \end{center}
\end{table}

%%---------------------------
\begin{figure}[ht]
\centering
\includegraphics[width=0.4\textwidth]{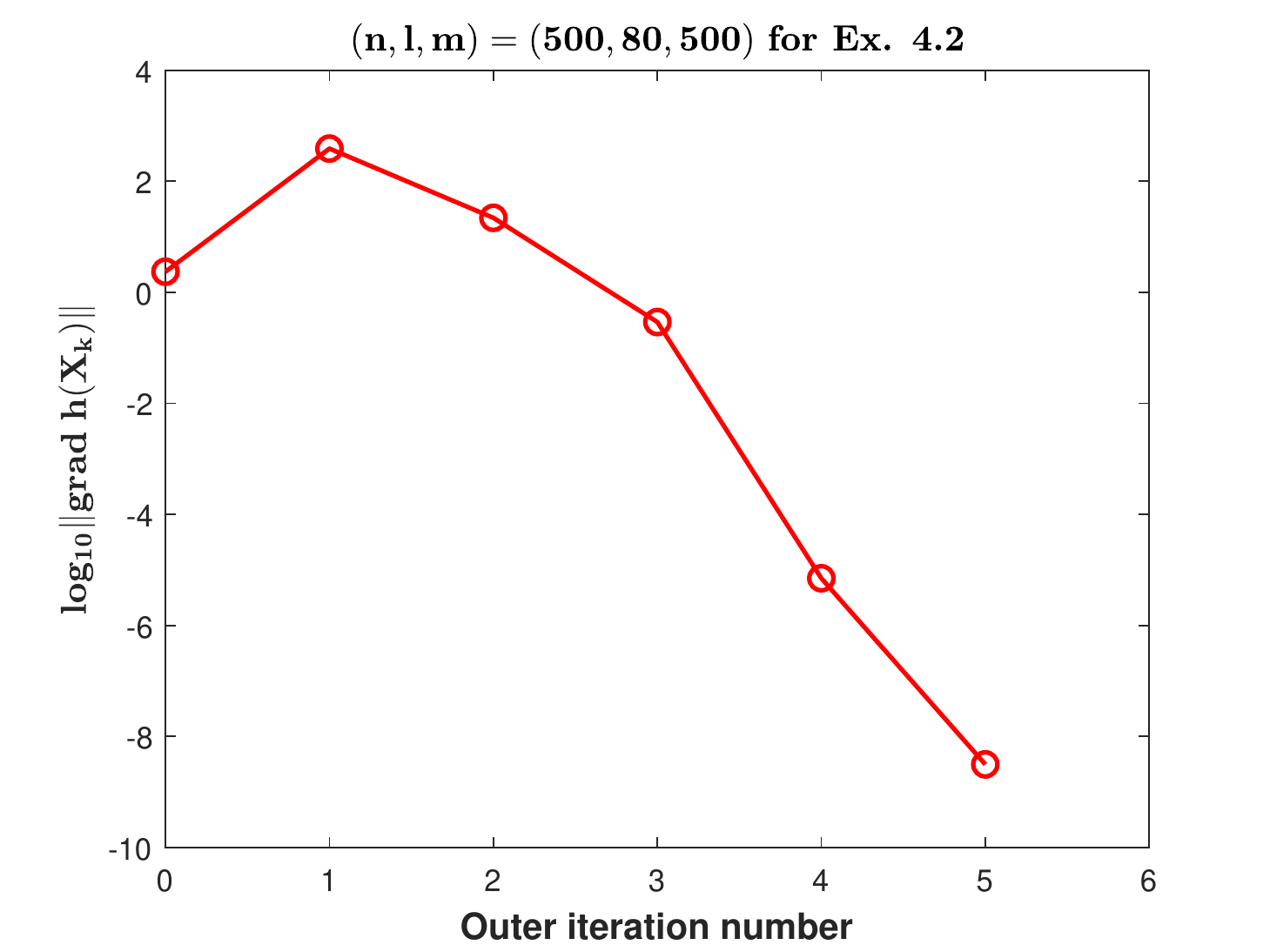}
\includegraphics[width=0.4\textwidth]{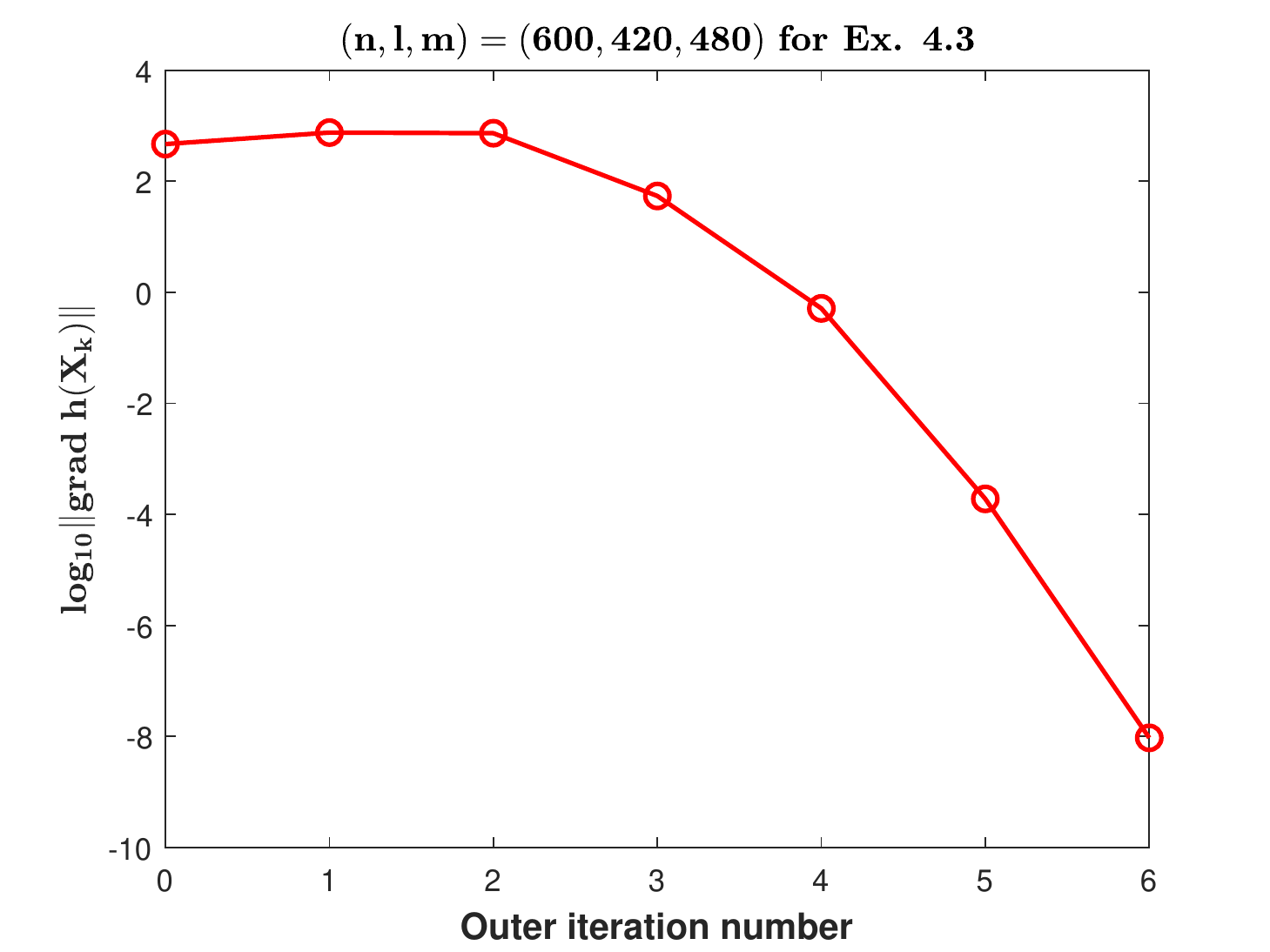}
\caption{Convergence history of two tests.}
 \label{fig}
\end{figure}

\section{Conclusions}\label{sec5}
In this paper, we have proposed a preconditioned Riemannian inexact Gauss-Newton method for solving the least squares inverse eigenvalue problem. The global and local convergence analysis of the method is established under some conditions. Numerical experiments show the efficiency of the proposed method.

%-----------------------------

\end{document}